\documentclass[11pt]{article}
\usepackage{amsmath,amssymb,amsthm}
\usepackage[usenames,dvipsnames]{color}
\usepackage{enumerate}
\usepackage{geometry}
\usepackage{graphicx}
\usepackage{subfigure}
\usepackage{epsfig}
\usepackage{graphicx}
\usepackage{colordvi}
\usepackage{graphics}
\usepackage{color}
\usepackage{float}

\numberwithin{equation}{section}
\newtheorem{theorem}[equation]{Theorem}

\newtheorem{lemma}[equation]{Lemma}
\theoremstyle{definition}
\newtheorem{remark}[equation]{Remark}
\newtheorem{algorithm}{Algorithm}[section]

\usepackage{mathtools}

\usepackage{fullpage} 

\newcommand{\dt}{{\Delta t}}

\newcommand{\chih}{\chi_h}

\newcommand{\phiv}{\phi_{h,v}}
\newcommand{\phiw}{\phi_{h,w}}
\newcommand{\etav}{\eta_{v}}
\newcommand{\etaw}{\eta_{w}}
\newcommand{\muo}{\mu}

\newcommand{\R}{\mathbb{R}}

\DeclareMathOperator{\rot}{rot} 

\newcommand{\ev}{e_v}
\newcommand{\ew}{e_w}

\usepackage[notref,notcite]{showkeys}

\usepackage[color=white,linecolor=black]{todonotes} 

\usepackage[square,comma,numbers,sort]{natbib}

\title{Continuous data assimilation applied to a velocity-vorticity formulation of the 2D Navier-Stokes equations}

\author{
Matthew Gardner\thanks{Department of Mathematical Sciences, Clemson University, Clemson, SC, 29634;
	email: mdgardn@g.clemson.edu}
\and
Adam Larios \thanks{Department of Mathematics, University of Nebraska-Lincoln, Lincoln, NE, 68588;
email: alarios@unl.edu, partially supported by NSF Grants DMS 1716801 and CMMI 1953346.}
\and
Leo G. Rebholz\thanks{Department of Mathematical Sciences, Clemson University, Clemson, SC, 29634;
email: rebholz@clemson.edu, partially supported by NSF Grant DMS 2011490.}
\and
Duygu Vargun
\thanks{Department of Mathematical Sciences, Clemson University, Clemson, SC, 29634;
	email: dvargun@clemson.edu.}
\and
Camille Zerfas
\thanks{Department of Mathematical Sciences, Clemson University, Clemson, SC, 29634;
email: czerfas@g.clemson.edu.}
}

\begin{document}
\date{}
\maketitle

\begin{abstract}
We study a continuous data assimilation (CDA) algorithm for a velocity-vorticity formulation of the 2D Navier-Stokes equations in two cases: nudging applied to the velocity and vorticity, and nudging applied to the velocity only.  We prove that under a typical finite element spatial discretization and backward Euler temporal
discretization, application of CDA preserves the unconditional long-time stability property of the velocity-vorticity method and provides
optimal long-time accuracy.  These properties hold if nudging is applied only to the velocity, and if nudging is also applied to the vorticity
then the optimal long-time accuracy is achieved more rapidly in time.  Numerical tests illustrate the theory, 
and show its effectiveness on an application problem of 
channel flow past a flat plate.
\end{abstract}

\section{Introduction}
Performing accurate simulations of {\color{black} complex} fluid flows that match real-world observations or experiments typically requires highly precise knowledge of the initial data.  However, such data is often known in very sparsely-distributed locations, which is the case in, e.g., weather observation,  ocean monitoring, etc. Thus, accurate, deterministic simulations based on initial data are often impractical.  \textit{Data assimilation} is a collection of methods that works around this difficulty by incorporating incoming data into the simulation to increase accuracy, hence data assimilation techniques are highly desirable to incorporate into simulations.  However, the underlying physical equations often suffer from stability issues which can reduce the accuracy gained by using data assimilation.  While there are many ways to stabilize numerical simulations, it is far from obvious how to adapt data assimilation techniques to combine them with cutting-edge stabilization methods.  Therefore it becomes worthwhile to seek new ways to incorporate data assimilation into stabilized schemes.  In this article, we propose and analyze a new approach to this problem which combines \textit{continuous data assimilation} with \textit{velocity-vorticity stabilization}.

Since Kalman's seminal paper \cite{Kalman_1960_JBE} in 1960, a wide variety of data assimilation algorithms have arisen (see, e.g., \cite{Daley_1993_atmospheric_book,Kalnay_2003_DA_book,Law_Stuart_Zygalakis_2015_book,Lewis_Lakshmivarahan_2008}).  In \cite{Azouani2014ContinuousDA}, Azouani, Olson, and Titi proposed a new algorithm known as continuous data assimilation (CDA), also referred to as the AOT algorithm.  Their approach revived the so-called ``nudging'' methods of the 1970's (see, e.g., \cite{Anthes_1974_JAS,Hoke_Anthes_1976_MWR}), but with the addition of a spatial interpolation operator.  This seemingly minor change had profound impacts, and the authors of \cite{Azouani2014ContinuousDA} were able to prove that using only sparse observations, the CDA algorithm applied to the 2D Navier-Stokes equations converges to the correct solution exponentially fast in time, independent of the choice initial data.  This stimulated a large amount of recent research on the CDA algorithm; see, e.g., 
\cite{Albanez_Nussenzveig_Lopes_Titi_2016,
Altaf_Titi_Knio_Zhao_Mc_Cabe_Hoteit_2015,
Bessaih_Olson_Titi_2015,
Biswas_Foias_Monaini_Titi_2018downscaling,
Biswas_Martinez_2017,
Carlson_Hudson_Larios_2018,
Celik_Olson_Titi_2019,
DiLeoni_Clark_Mazzino_Biferale_2018_unraveling,
DiLeoni_Clark_Mazzino_Biferale_2019,
Farhat_GlattHoltz_Martinez_McQuarrie_Whitehead_2019,
Farhat_Johnston_Jolly_Titi_2018,
Farhat_Jolly_Titi_2015,
Farhat_Lunasin_Titi_2016abridged,
Farhat_Lunasin_Titi_2016benard,
Farhat_Lunasin_Titi_2016_Charney,
Farhat_Lunasin_Titi_2017_Horizontal,
Farhat_Lunasin_Titi_2018_Leray_AOT,
Foias_Mondaini_Titi_2016,
Foyash_Dzholli_Kravchenko_Titi_2014,
GarciaArchilla_Novo_Titi_2018,
Gesho_Olson_Titi_2015,
GlattHoltz_Kukavica_Vicol_2014,
Ibdah_Mondaini_Titi_2018uniform,
Jolly_Martinez_Olson_Titi_2018_blurred_SQG,
Jolly_Martinez_Titi_2017,
Larios_Pei_2017_KSE_DA_NL,
Larios_Victor_2019,
LRZ19,
Lunasin_Titi_2015,
Markowich_Titi_Trabelsi_2016,
Mondaini_Titi_2018_SIAM_NA,
Pei_2019,
RZ19}.  {\color{black} The recent paper \cite{DHKTLD19} showed that CDA can be effectively used for weather prediction, showing that it can indeed be a powerful tool on practical large scale problems.  Convergence of discretizations of CDA models was studied in \cite{LRZ19,RZ19,IMT18,GNT18} , and found results similar to those at the continuous level.  } Our interest in the CDA algorithm arises from its adaptability to a wide range of nonlinear problems, as well as its small computational cost and straight-forward implementation.  These qualities make it an ideal candidate for combining data assimilation with stabilization techniques; in particular, with the {\color{black} recently developed} velocity-vorticity stabilization, described below.

Flows of incompressible, viscous Newtonian fluids are modeled by the Navier-Stokes equations (NSE), which take the form
\begin{equation}\label{NSE}
\begin{aligned}
u_t - \nu\Delta u + (u\cdot \nabla) u + \nabla p =& f,\\
\nabla\cdot u =& 0,
\end{aligned}
\end{equation}
together with suitable boundary and initial conditions.  Here, $u$ denotes a velocity vector field, $p$ is pressure, $f$ is external (given) force, and $\nu>0$ represents the kinematic viscosity which is inversely proportional to the Reynolds number.  Solving the NSE is important in many applications, however it is well known that doing so can be quite difficult, especially for small $\nu$.  Many different tools
have been used for more accurate numerical simulations of the NSE, for example using NSE formulations tailored to particular application problems \cite{GS98,CHOR17,OR10,LMNOR09} or discretization and stabilization methods \cite{Laytonbook,OR04,J16}, and more recently using observed data to improve simulation \cite{Azouani2014ContinuousDA,LRZ19,czphdthesis,ZRSI19,Biswas_Hudson_Larios_Pei_2017}. 

We consider in this paper discretizations of a continuous data assimilation (CDA) enhancement applied to the following velocity-vorticity (VV) formulation of the 2D NSE:
\begin{equation}\label{VV}
\begin{aligned}
u_t - \nu \Delta u + \omega \times u + \nabla P = f,\\
\nabla\cdot u=0,\\
\omega_t- \nu \Delta \omega + (u\cdot\nabla)\omega=\rot f.
\end{aligned}
\end{equation}
Here, $\omega$ represents the (scalar) vorticity, $P:=p+\tfrac12|u|^2$ is the Bernoulli pressure, and rot is the 2D curl operation: $\rot \binom{f_1}{f_2}:=\tfrac{\partial f_1}{\partial y} - \tfrac{\partial f_2}{\partial x}$.  In the NSE, the velocity and vorticity are coupled via the relationship $\omega=\rot u$ (or equivalently, the Biot-Savart Law).  However, the VV formulation typically does \textit{not} enforce this relationship, so $u$ and $\omega$ are only coupled via the evolution equations in \eqref{VV}, and the relationship $\omega=\rot u$ is  recovered \textit{a posteriori}, so that at the continuous level, \eqref{VV} is formally equivalent to \eqref{NSE}.  However, in practice, discretizations of VV can behave quite differently from {\color{black} typical} discretizations of NSE, providing better stability as well as accuracy (especially for vorticity) for vortex dominated or strongly rotating flows, see \cite{Olshanskii2015NaturalVB,Olshanskii2010VelocityvorticityhelicityFA,articleLEE,articleAkbas} and references therein. A very interesting property of \eqref{VV} was recently shown in \cite{HOR17}, where it was proven that the system \eqref{VV} when discretized with standard finite elements and a decoupling backward Euler or BDF2 temporal discretization was unconditionally long-time stable in both $L^2$ and $H^1$ norms for both velocity and vorticity; no such analogous result is known for velocity-pressure discretizations/schemes.  {\color{black} Hence the scheme itself is stabilizing, even though it is still formally consistent with the NSE.}  The recent work in \cite{articleAkbas} showed that these unconditionally long-time stable schemes also provide optimal vorticity accuracy, yielding a vorticity solution that is one full order of spatial accuracy better than for an analogous velocity-pressure scheme.

{\color{black}
 We consider herein CDA applied to \eqref{VV}, which yields a model of the form
\begin{equation}\label{VVda}
\begin{aligned}
v_t - \nu \Delta v + w \times v + \nabla q + \mu_1 I_H (v - u)= f,\\
\nabla\cdot v=0,\\
w_t- \nu \Delta w + (v\cdot\nabla)w + \mu_2 I_H(w -\omega) =\rot f,
\end{aligned}
\end{equation}
where $I_H$ is an appropriate interpolation operator, $I_H(u)$ and $I_H(\omega)$ are assumed known from measurements, and
$\mu_1,\mu_2\ge 0$ are nudging parameters.  If $\mu_2=0$, then vorticity is not nudged and $I_H(\omega)$ need not be assumed known.
Due to the success of \eqref{VV} in recent papers \cite{HOR17,articleAkbas,Olshanskii2015NaturalVB} and that of CDA in the works mentioned above, combining these ideas 
and studying \eqref{VV} is a natural next step to see whether CDA will provide
optimal long-time accuracy for the VV schemes already known to be unconditionally long-time stable.  Herein, we do find
that CDA provides convergence of \eqref{VVda}, with any initial condition, to the true NSE solution (up to optimal discretization error) and moreover that CDA 
preserves the long-time stability.
}

This paper is organized as follows. In Section \ref{prelim}, we introduce the necessary notation and preliminaries needed in the analysis. In Section 3, we propose and analyze a fully discrete scheme for \eqref{VVda}, and show that for nudging velocity and vorticity together and nudging just velocity, algorithms are long-time stable in $L^2$ and $H^1$ norms and long-time optimally accurate in $L^2$ velocity and vorticity (under the usual CDA assumptions on the coarse mesh and nudging parameter).
In Section 4, we illustrate the theory with numerical tests, and finally draw conclusions in section 5.

\section{Notation and Preliminaries}\label{prelim}

We now provide notation and mathematical preliminaries to allow for a smooth analysis to follow.  {\color{black} We consider the domain $\Omega \subset \R^2$ to be the $2\pi$-periodic box}, with the $L^2(\Omega)$ norm and inner product denoted by $\| \cdot \|$ and $(\cdot, \cdot)$ respectively, while all other norms will be appropriately labeled. 


{\color{black}
For simplicity, we use herein periodic boundary conditions for velocity and vorticity.  Extension to full nonhomogeneous Dirichlet conditions can
be performed by following analysis in \cite{articleLEE}, although for no-slip velocity together with the more physically consistent  natural vorticity boundary condition
studied in \cite{Olshanskii2015NaturalVB,ORS18} more work would be needed to handle the boundary integrals.  We denote the natural corresponding function spaces for velocity, pressure, and vorticity by
\begin{align*}
X & := H^1_{\#}(\Omega)^2 = \left\{ v\in H^1_{loc}(\mathbb{R})^2,\ v \mbox{ is $2\pi$-periodic in each direction},\ \int_{\Omega} v\ dx=0 \right\},  \\
Q & := L^2_{\#}(\Omega) = \left\{ q\in L^2_{loc}(\mathbb{R}),\ q \mbox{ is $2\pi$-periodic in each direction},\ \int_{\Omega} q\ dx=0 \right\},  \\
W &:= H^1_{\#}(\Omega)  = \left\{ v\in H^1_{loc}(\mathbb{R}),\ v \mbox{ is $2\pi$-periodic in each direction},\ \int_{\Omega} v\ dx=0 \right\}.
\end{align*}
}

In $X$ (and $W$), we have the Poincar\'e inequality: there exists a constant $C_P$ depending only on $\Omega$ such that for any $\phi\in X$ (or $W$),
\[
\| \phi \| \le C_P \| \nabla \phi \|.
\]

We define the skew-symmetric trilinear operator $b^*:X\times W \times W \rightarrow \mathbb{R}$ to use for the nonlinear term in the vorticity equation, by
\begin{align*}
b^*(u,\omega,\chi):=\frac{1}{2}\left((u\cdot\nabla \omega, \chi) - (u\cdot\nabla \chi, \omega)\right).
\end{align*}

The following lemma is proven in \cite{LRZ19}, and is useful in our analysis.
\begin{lemma}\label{geoseries}
	Suppose constants $r$ and $B$ satisfy $r>1$, $B\ge 0$.  Then if the sequence of real numbers $\{a_n\}$ satisfies 
	\[
	ra_{n+1} \le a_n + B,
	\]
	we have that
	\[
	a_{n+1} \le a_0\left(\frac{1}{r}\right)^{n+1}  + \frac{B}{r-1}.
	\]
\end{lemma}

\subsection{Discretization preliminaries}

Denote by $\tau_h$ a regular, conforming triangulation of the domain $\Omega$, and let $X_h \subset X$, $Q_h \subset Q$ be velocity-pressure spaces that satisfy the inf-sup condition. We will assume the use of  $X_h = X \cap P_k(\tau_h)$ and $Q_h = Q \cap P_{k-1}(\tau_h)$ Taylor-Hood or Scott-Vogelius elements (on appropriate meshes and/or polynomial degrees, see \cite{GS19} and references therein).  The discrete vorticity space is defined as $W_h := W \cap P_k(\tau_h).$ Define the discretely divergence free subspace by 
\[ 
V_h := \{ v_h \in X_h \,\, |\,\, (\nabla \cdot v_h, q_h) = 0 \,\, \forall \,\, q_h \in Q_h   \} . 
\]

We will assume the mesh is sufficiently regular so that the inverse inequality holds in $X_h$: There exists a constant $C$ such that
\begin{align*}
	\|\nabla \chi_h\| &\leq C h^{-1}\|\chi_h\| \quad \forall \,\, \chi_h \in X_h.
\end{align*} 
The discrete Laplacian operator is defined as: For $\phi\in H^1(\Omega)^2$, $\Delta_h \phi \in X_h$ satisfies  
\begin{align}
(\Delta_h \phi, v_h) \,\,=\,\, - (\nabla \phi, \nabla v_h) \ \forall v_h \in X_h.\label{laplace}
\end{align}
The definition for $\Delta_h$ is written the same way when applied in $W_h$, since this is simply the above definition
restricted to a single component.  

The discrete Stokes operator $A_h$ is defined as: For $\phi\in H^1(\Omega)^2$, find $A_h \phi \in V_h$ such that for all $v_h \in V_h$, 
\begin{align}
 (A_h \phi, v_h) \,\,=\,\, - (\nabla \phi, \nabla v_h).\label{stokes}
\end{align}
By the definition of discrete Laplace and Stokes operators, we have the Poincar\'e inequalities
\begin{align}
\|\nabla \chi_h\|\leq C_P \|\Delta_h \chi_h\|\ \forall \chi_h\in X_h,\label{Pforlap}
\\
\|\nabla \phi_h\|\leq C_P \|A_h \phi_h\|\ \forall \phi_h\in V_h.\label{Pforsto}
\end{align}
We recall the following discrete Agmon inequalities and discrete $L^p$ bounds \cite{HOR17,J16}:
\begin{align}
\|v_h\|_{L^{\infty}}& \leq C \|v_h\|^{1/2} \|A_h v_h\|^{1/2}\ \forall v_h\in V_h,\label{agmonlap}
\\
\|v_h\|_{L^{\infty}}& \leq C \|v_h\|^{1/2} \|\Delta_h v_h\|^{1/2}\ \forall v_h\in X_h,\label{agmonstk} \\
\|\nabla v_h\|_{L^{3}}& \leq C \|v_h\|^{1/3} \|\Delta_h v_h\|^{2/3}\ \forall v_h\in X_h. \label{l3ineq}
\end{align}
We note that all bounds above for $X_h$ trivially hold in $W_h$, since $W_h$ functions can be considered as components of functions in $X_h$.

A function space for measurement data interpolation is also needed.  Hence we require another regular conforming mesh $\tau_H$, and define 
$X_H = P_r(\tau_H)^2$ and $W_H = P_r(\tau_H)$ for some polynomial degree $r$.  We require that the coarse mesh interpolation operator $I_H$ used for data assimilation satisfies the following bounds:  for any $w \in H^1(\Omega)^d$,
\begin{align}
	\| I_H  (w) - w \| &\le C H \| \nabla w\|, \label{interp1}
	\\ \| I_H(w) \| &\leq  C \|w \|. \label{interp2}
\end{align}
These are key properties for the interpolation operator that allow for both mathematical theory as well as providing guidance on how small $H$ should be (i.e. how many measurement points are needed).  We note the same $I_H$ operator is used for vector functions and scalar functions, with it being applied component-wise for vector functions.

\section{Analysis of a CDA-VV scheme }\label{analysisDA}

We consider now a discretization of \eqref{VVda} that uses a finite element spatial discretization and backward Euler temporal discretization.  The backward
Euler discretization is chosen only for simplicity of analysis; all results extend to the analogous BDF2 scheme following analysis in \cite{articleAkbas,HOR17}.  One difference of our scheme below compared to other discretizations of CDA is that $I_H$ is also applied to the test functions in the nudging terms.  This was first proposed by the authors in \cite{RZ19}, and allows for a simpler stability analysis as well as to the use of special types of efficient interpolation operators.

	\begin{algorithm} \label{VVNLbe}
		Given $v_h^0\in V_h$ and $w_h^0\in W_h$, find $(v_h^{n+1}, w_h^{n+1}, P_h^{n+1}) \in (X_h, W_h, Q_h)$ for $n = 0,1,2,...$, satisfying 
		\begin{align}
		\frac{1}{\Delta t} \left( v_h^{n+1} - v_h^n,\chi_h \right) + (w_h^{n} \times v_h^{n+1}, \chi_h) - (P_h^{n+1},\nabla \cdot \chi_h) + \nu (\nabla v_h^{n+1}&,\nabla \chi_h)   \nonumber \\ + \muo_1 (I_H(v_h^{n+1} - u^{n+1}),I_H(\chi_h))&=  (f^{n+1},\chi_h), \label{nlda5} \\
		(\nabla \cdot v_h^{n+1},r_h)  &= 0, \label{nlda6}
		\\ \frac{1}{\Delta t} \left( w_h^{n+1} - w_h^n,\psi_h \right) +b^*(v_h^{n+1}, w_h^{n+1},\psi_h)  + \nu (\nabla w_h^{n+1},\nabla \psi_h) \nonumber\\
		 	+ \muo_2 (I_H(w_h^{n+1} - \rot u^{n+1}),I_H(\psi_h)) 
		&=  (\rot f^{n+1},\psi_h), \label{nlda7}
		\end{align}
		for all $(\chi_h, \psi_h, r_h) \in (X_h,W_h,Q_h)$, where  $I_H(u^{n+1})$, $I_H(\rot u^{n+1})$ are assumed known for all $n\ge 1$.
	\end{algorithm}
	We begin our analysis with long-time stability estimates, followed by long-time accuracy.
	
\subsection{Stability analysis of Algorithm \ref{VVNLbe}}

In this subsection, we prove that Algorithm \ref{VVNLbe} is unconditionally long-time $L^2$ and $H^1$ stable for both velocity and vorticity.  This property was proven for the scheme without nudging in \cite{HOR17}, and so these results show that CDA preserves this important property that is (seemingly) unique to VV schemes of this form.



\begin{lemma}[$L^2$ stability of velocity and vorticity ] \label{L2stabboth}
	Let $f \in L^\infty(0,\infty; L^2)$ and $u\in L^{\infty}(0,\infty; H^1)$. Then, for any $\dt>0$, any integer $n>0$, and nudging parameters $\mu_1,\mu_2\ge 0$, velocity and vorticity solutions to Algorithm \ref{VVNLbe} satisfy
	\begin{align}
		\|v_h^{n}\|^2 
	\leq&
	\alpha^{-n}\|v_h^{0}\|^2 
	+
	\frac{C C_P^{2}}{\nu}
	(\nu^{-1} \|f\|^2_{L^{\infty}(0,\infty;H^{-1})}
	+
	\mu_1 \|u\|^2_{L^{\infty}(0,\infty;L^2)})
	=:
	C_1,\label{L2stabvel}
	\\
	\|\omega_h^{n}\|^2
	\leq&
	\alpha^{-n}\|\omega_h^{0}\|^2
	+
	\frac{C C_P^{2}}{\nu } (\nu^{-1}\|f\|^2_{L^{\infty}(0,\infty;L^2)}
	+
	\mu_2 \|\rot u\|^2_{L^{\infty}(0,\infty;L^2)})
	=: C_2,\label{L2stabvortmu}
	\end{align}
	where $\alpha=1+\nu C_P^{-2}\Delta t$.
\end{lemma}




\begin{proof}
	Begin by choosing $\chi_h=2\Delta t v_h^{n+1}$ in \eqref{nlda5}, which vanishes the nonlinear and pressure terms, and leaves
	\begin{align*}
	\|v_h^{n+1}\|^2 - \|v_h^{n}\|^2 + \|v_h^{n+1}-v_h^{n}\|^2
	+ 
	2\Delta t\nu \|\nabla v_h^{n+1}\|^2 
	+ 
	2\Delta t\muo_1  \|I_H(v_h^{n+1})\|^2
	\nonumber\\
	=
	2\Delta t(f^{n+1}, v_h^{n+1})
	+
	2\Delta t \muo_1 (I_H(u^{n+1}),I_H(v_h^{n+1})).
	\end{align*}
	The first right hand side term is bounded using the dual norm and Young's inequality via
	\begin{align*}
	2\Delta t(f^{n+1}, v_h^{n+1})\leq C\Delta t\nu^{-1}\|f^{n+1}\|^2_{-1} +\Delta t \nu\|\nabla  v_h^{n+1} \|^2,
	\end{align*}
	and for the interpolation term, we use Cauchy-Schwarz, the interpolation property (\ref{interp2}) and Young's inequality to get
	\begin{align*}
	2\Delta t\mu_1 (I_H(u^{n+1}),I_H(v^{n+1}_{h}))
	\leq&
	2\Delta t\mu_1 \|I_H(u^{n+1})\| \|I_H(v^{n+1}_{h})\|
	\nonumber\\
	\leq&
	C\Delta t\mu_1  \|u^{n+1}\| \|I_H(v^{n+1}_{h})\|
	\nonumber\\
	\leq&
	C \Delta t\mu_1  \|u^{n+1}\|^2 + \Delta t\mu_1 \|I_H(v^{n+1}_{h})\|^2.
	\end{align*}
	Combining the above estimates and dropping $\|v_h^{n+1}-v_h^{n}\|^2$ and $\|I_H(v^{n+1}_{h})\|^2$ from the left hand side produces the bound
	\begin{align*}
	\|v_h^{n+1}\|^2 
	+ 
	\Delta t \nu \|\nabla v_h^{n+1}\|^2 
	\leq
	\|v_h^{n}\|^2 
	+
	C\Delta t\nu^{-1}\|f^{n+1}\|^2_{-1} 
	+
	C\Delta t\mu_1 \ \|u^{n+1}\|^2,
	\end{align*}
	and thanks to the Poincar\'e inequality, we obtain
	\begin{align*}
	(1+ C_P^{-2} \Delta t \nu)\|v_h^{n+1}\|^2 
	\leq
	\|v_h^{n}\|^2 
	+
	C\Delta t \left( \nu^{-1}\|f^{n+1}\|^2_{-1} 
	+
	\mu_1\|u^{n+1}\|^2\right).
	\end{align*}
	Defining $\alpha=1+ C_P^{-2} \Delta t \nu> 1$, and then applying Lemma \ref{geoseries} reveals the $L^2$ stability bound (\ref{L2stabvel}) for the velocity solution of Algorithm \ref{VVNLbe}.
	
	Applying similar analysis to the above will produce the stated $L^2$ vorticity bound.

\end{proof}

\begin{lemma}[$H^1$ stability of velocity and vorticity ]\label{h1stanVVNLbothbe}
	Let $f \in L^\infty(0,\infty; H^1)$ and $u \in L^{\infty}(0,\infty; H^1)$. Then, for any $\dt>0$, any integer $n>0$, and nudging parameters $\mu_1,\mu_2\ge 0$, velocity and vorticity solutions to  Algorithm \ref{VVNLbe} satisfy
	\begin{align}
		\|\nabla v_h^{n}\|^2
	\leq&
	\alpha^{-n}
	\|\nabla v_h^{0}\|^2
	+
	\frac{C C_P^{2}}{\nu^2}
	\left(
	\|f\|^2_{L^{\infty}(0,\infty;L^2)}	
	+
	C_2^4 C_1^{2} \nu^{-2}
	+
	\mu_1^2\| u \|^2_{L^{\infty}(0,\infty;L^2)}
	+
	\mu_1^2   C_1^2
	\right)=: \tilde{C}_1,\label{H1stabvel}
	\\
	\|\nabla w_h^{n}\|^2 
	\leq&
	\alpha^{-n} \|\nabla w_h^{0}\|^2
	+
	\frac{C^2_P C}{\nu^2}
	\left(
	|\rot f^{n+1}\|^2 
	+
	 \nu^{-4} \tilde{C}_1^6 C_2^{2}
	+
	\nu^{-2} \tilde{C}_1^4 C_2^{2}
	+
	 \mu_2^2 \| \rot u^{n+1} \|^2 
	+ 
	\mu_2^2  C_2^2 
	\right),\label{H1stabvortmu}
	\end{align}
	where $\alpha=1+\nu C_P^{-2}\Delta t$.
\end{lemma}
\begin{proof}
	After testing the velocity equation (\ref{nlda5}) with $\chi_h = 2\dt A_hv_h^{n+1}$, we obtain 
	\begin{align*}
	&\|\nabla v_h^{n+1}\|^2 
	-
	\|\nabla v_h^n\|^2 
	+
	\|\nabla(v_h^{n+1}-v_h^{n})\|^2
	+
	2\dt\nu \|A_hv_h^{n+1}\|^2
	\nonumber\\
	&\leq 
	2\dt (f^{n+1},A_h v_h^{n+1}) 
	+
	2\dt |(w_h^{n}\times v_h^{n+1}, A_hv_h^{n+1})| 
	+
	2\dt \mu_1 (I_H(u^{n+1}-v_h^{n+1}), I_H(A_h v_h^{n+1})).
	\end{align*}
	We now bound the right hand side terms. First, the forcing term is bounded by Cauchy-Schwarz and Young's inequalities via
	\begin{align}\label{fw}
	2\dt (f^{n+1},A_h v_h^{n+1}) 
	\leq
	C\dt \nu^{-1}\|f^{n+1}\|^2
	+
	\frac{\nu}{4}\dt\|A_h v^{n+1}\|^2.
	\end{align}
	Then, for the nonlinear terms, we again apply H\"older, discrete Agmon (\ref{agmonstk}) and generalized Young inequalities, and the result of Lemma \ref{L2stabboth} to get
	\begin{align*}
	2\dt |(w_h^{n}\times v_h^{n+1}, A_hv_h^{n+1})| 
	\leq&
	2\dt \|w_h^{n}\| \|v_h^{n+1}\|_{L^\infty} \|A_hv_h^{n+1}\|
	\nonumber\\
	\leq&
	 C\dt \|w_h^{n}\| \|v_h^{n+1}\|^{1/2} \|A_hv_h^{n+1}\|^{3/2}
	\nonumber\\
	\leq&
	C\dt\nu^{-3}   \|w_h^{n}\|^4 \|v_h^{n+1}\|^{2} 
	+
	\frac{\nu}{8}\dt \|A_hv_h^{n+1}\|^{2}
	\nonumber
	\\
	\leq&
	 C C_2^4 C_1^{2} \dt \nu^{-3}
	+
	\frac{\nu}{8}\dt \|A_hv_h^{n+1}\|^{2}.
	\end{align*}
	Lastly, the interpolation term is bounded using Cauchy-Schwarz and interpolation property \ref{interp2},  followed by Young's inequality and the result of Lemma \ref{L2stabboth} to obtain
	\begin{align}\label{intpstab}
	2\dt \mu_1 (I_H&(u^{n+1}-v_h^{n+1}), I_H(A_h v_h^{n+1}))
	\nonumber\\&
	\leq
	2\dt \mu_1\left(|(I_H u^{n+1}, I_H A_h v_h^{n+1})|+|(I_H v_h^{n+1}, I_H A_h v_h^{n+1})|\right)
	\nonumber\\
	&\leq
	C\dt \mu_1\|I_H u^{n+1} \| \|I_H A_h v_h^{n+1} \|+C\dt \mu_1\|I_H v_h^{n+1} \| \|I_H A_h v_h^{n+1}\|
	\nonumber	\\
	&\leq
	C\dt \mu_1^2 \nu ^{-1} \| u^{n+1} \|^2 + C\dt \mu_1^2 \nu^{-1}  C_1^2 + \frac{\nu}{8} \dt \| A_h v_h^{n+1}\|^2.
	\end{align}
	Combining all these bounds for right hand side terms and dropping nonnegative term  $\|\nabla(v_h^{n+1}-v_h^{n})\|^2 $ on left hand side give us that
	\begin{align*}
	\|\nabla v_h^{n+1}\|^2
	+
	\dt\nu \|A_hv_h^{n+1}\|^2 
	&\leq 
	\|\nabla v_h^{n}\|^2
	+
	C \dt \nu^{-1}\|f^{n+1}\|^2
	+
	\nu^{-3} \dt C C_2^4 C_1^{2} 
	+
	\dt \mu_1^2\nu^{-1} C \|u^{n+1}\|^2
	\\ & \,\,\,\,
	+ C\dt \mu_1^2 \nu^{-1}  C_1^2.
	\end{align*}	
	By the Poincar\'e inequality \eqref{Pforsto}, we now get
	\begin{align*}
	\alpha\|\nabla v_h^{n+1}\|^2\leq \|\nabla v_h^{n}\|^2+ C\dt \bigg( \nu^{-1}\|f^{n+1}\|^2
	+
	\nu^{-3}  C_2^4 C_1^{2} 
	+
	 \mu_1^2\nu^{-1}  \|u^{n+1}\|^2
	+  \mu_1^2 \nu^{-1}  C_1^2 \bigg),
	\end{align*}
	where $\alpha=1+\nu C_P^{-2} \Delta t$. Finally, we apply Lemma \ref{geoseries} and reveal (\ref{H1stabvel}).
	
	For the vorticity estimate, choose $\psi_h = 2\dt \Delta_h w_h^{n+1}$ in  (\ref{nlda7}) to get
	\begin{align*}
	&\|\nabla w_h^{n+1}\|^2 
	-
	\|\nabla w_h^{n}\|^2
	+
	\|\nabla (w_h^{n+1}-w_h^n)\|^2 
	+
	2\dt \nu \|\Delta_hw_h^{n+1}\|^2 
	\nonumber 
	\\  
	&\leq
	2\dt |(\rot f^{n+1}, \Delta_h w_h^{n+1})| 
	+
	2\dt |b^*(v_h^{n+1}, w_h^{n+1}, \Delta_h w_h^{n+1})|
	+ 
	2\dt\muo_2(I_H(\rot u^{n+1}-w_h^{n+1}),I_H(\Delta_h w_h^{n+1})). 
	\end{align*}
	From here, the proof follows the same strategy as the $H^1$ velocity proof above, except the nonlinear term is handled slightly differently.
%
%
	We use the discrete Agmon inequality (\ref{agmonstk}), the discrete Sobolev inequality (\ref{l3ineq}), the result of Lemma \ref{L2stabboth}, the $H^1$ stability bound for vorticity (\ref{H1stabvel}) proven above,  and the generalized Young's inequality, as follows.
	\begin{align*}
	2\dt |b^*&(v_h^{n+1}, w_h^{n+1}, \Delta_hw_h^{n+1})|
	\\
	&\leq
	2\dt
	\left(
	|(v_h^{n+1}\cdot\nabla w_h^{n+1}, \Delta_hw_h^{n+1})|
	+
	\frac{1}{2}
	|((\nabla\cdot v_h^{n+1}) w_h^{n+1}, \Delta_hw_h^{n+1})|
	\right)
	\\
	&\leq
	2\dt \|v_h^{n+1}\|_{L^6} \|\nabla w_h^{n+1}\|_{L^3} \|\Delta_hw_h^{n+1}\|
	+
	\dt
	\|\nabla v_h^{n+1}\| \|w_h^{n+1}\|_{L^{\infty}} \|\Delta_hw_h^{n+1}\|
	\\
	&\leq
	C\dt \tilde{C}_1 \| w_h^{n+1}\|^{1/3} \|\Delta_h w_h^{n+1}\|^{5/3}
	+
	C\dt
	\tilde{C}_1 \|w_h^{n+1}\|^{1/2} \|\Delta_h w_h^{n+1}\|^{3/2}
	\\
	&\leq
	C\dt \tilde{C}_1 C_2^{1/3} \|\Delta_h w_h^{n+1}\|^{5/3}
	+
	C\dt
	\tilde{C}_1 C_2^{1/2} \|\Delta_h w_h^{n+1}\|^{3/2}
		\\
	&\leq
	C\dt \nu^{-5} \tilde{C}_1^6 C_2^{2}
	+
	C\dt \nu^{-3}
	\tilde{C}_1^4 C_2^{2}
	+
	\frac{\nu}{3}\dt
	\|\Delta_h w_h^{n+1}\|^{2}.
	\end{align*}
	Now proceeding as in the velocity $H^1$ bound will produce the $H^1$ vorticity stability bound \eqref{H1stabvortmu}.
%
\end{proof}

\subsection{Long-time accuracy of Algorithm \ref{VVNLbe} }
We now consider the difference between the solutions of \eqref{nlda5} - \eqref{nlda7} to the NSE solution. We will show that the algorithm solution converges to the true solution, up to an optimal $O(\dt + h^{k+1})$ discretization error, independent of the initial condition, provided a restriction on the coarse mesh width and nudging parameters.  We will give two results, the first for $\mu_2>0$ and the second for $\mu_2=0$; while they both provide optimal long-time accuracy, when $\mu_2>0$ the convergence to the true solution occurs more rapidly in time.

In our theory below for long-time accuracy of Algorithm \ref{VVNLbe}, we assume the use of Scott-Vogelius elements. This is done for simplicity, as for non-divergence-free elements like Taylor-Hood elements, similar optimal results can be obtained (although with some additional terms and different constants) but require more technical details; see, e.g., \cite{GNT18}.

\begin{theorem}[Long-time $L^2$ accuracy of Algorithm \ref{VVNLbe} with $\muo_1>0$ and $\muo_2>0$] \label{nlthmL2}
	Let true solutions $u \in L^\infty(0, \infty ; H^{k+2}(\Omega))$, $ p \in L^\infty(0, \infty ; H^{k}(\Omega))$ where $k\geq 1$ and $u_{t}, u_{tt}\in L^\infty (0, \infty ; H^1)$, and we assume properties of the domain permits optimal $L^2$ and $H^1$ accuracy of the discrete Stokes projection in $V_h$ and discrete $H^1_0$ projection into $W_h$. Then, assume that time step $\dt$ is sufficiently small, and that $\mu_1$ and $\muo_2$ satisfy
$$\max\left\{1, C\nu^{-1}\left(\|\omega^{n+1}\|^2_{L^{\infty}}+\|\etaw^{n+1}\|^2_{L^{\infty}} \right) \right\}\leq \mu_1 \leq \frac{C\nu}{H^2},$$
$$\max\left\{1,C\nu^{-1}\left(\|u^{n+1}\|^2_{L^{3}}+\|\etav^{n+1}\|^2_{L^{3}} \right) \right\}\leq \mu_2 \leq \frac{C\nu}{H^2},$$
where $H$ is chosen so that this inequality holds. Then, for any time $t^n$, $n = 0, 1, 2,...$, we have for solutions of Algorithm \ref{VVNLbe} using Scott-Vogelius elements,
\begin{align*}
\|v_h^n-u^n\|^2 + \|\omega_h^n - \rot u^n\|^2 & \leq
(1 + \lambda \dt )^{-n}(\|v_h^0-u^0\|^2 + \|\omega_h^0 - \rot u^0\|^2)
+ C \lambda^{-1} R,
\end{align*}
where
	\[R:=\left(\mu_1^{-1} \dt^2 + \mu_2^{-1} \dt^2 + \nu^{-1} h^{2k+2} + \mu_1 h^{2k+2} + \mu_2 h^{2k+2}\right),   \]
	and
 $\lambda=\min \left\{ \frac{\mu_1}{4} + \frac{\nu C_P^{-2}}{4} , \frac{\mu_2}{4} +  \frac{\nu C_P^{-2}}{4} \right\}$ with $C$ independent of $\dt$, $h$ and $H$. 
\end{theorem}

\begin{proof}
	The true NSE solution satisfies the VV system
	\begin{align*}
	\frac{1}{\dt} (u^{n+1} - u^n) + \omega^{n} \times u^{n+1} + \nabla P^{n+1} - \nu \Delta u^{n+1} & = f^{n+1} - \dt u_{tt}(t^*) + (\omega^{n} - \omega^{n+1} )\times u^{n+1} ,
	\\ \nabla \cdot u^{n+1} & = 0,
	\\ \frac{1}{\dt}(\omega^{n+1} - \omega^n) + u^{n+1}\cdot \nabla \omega^{n+1} - \nu \Delta \omega^{n+1} &= \rot f^{n+1} - \dt\omega_{tt}(t^{**}),
	\end{align*}
	where $u^n$ is the velocity at time $t^n$, $P^n$ the Bernoulli pressure, $\omega^n := \rot u^n$, and $t^*, t^{**} \in [t^n, t^{n+1}]$. Note that by Taylor expansion, we can write $\omega^{n}-\omega^{n+1}=-\dt \omega_t (s^*)$ where $s^* \in [t^n, t^{n+1}]$. 
	
	The difference equations are obtained by subtracting the solutions to Algorithm \ref{VVNLbe} from the NSE solutions by defining the differences between velocity and vorticity as $e_v^{n} := u^n - v_h^n$ and $e_w^n := \omega^n - w_h^n$, respectively. Next, we will decompose the error into a term that lies in the discrete space $V_h$ and one outside. To do so, add and subtract the discrete Stokes projection of $u^n$, denoted $s_h^n$, to $e_v^n$ and let $\eta_v^n := s_h^n - u^n$, $\phi_{h,v}^n := v_h^n - s_h^n$. Then $e_v^n = \phi_{h,v}^n + \eta_v^n$ and $\phiv^{n} \in V_h$. In a similar manner, by taking the $H^1_0$ projection of $\rot u^n$ into $W_h$, we obtain $e_w^n = \phi_{h,w}^n + \eta_w^n$ with $\phiw^n \in V_h$. 

	  For velocity, since $(\nabla \eta_v^{n+1},\nabla \phi_{h,v}^{n+1})=0$, the difference equation becomes
		\begin{align}
	\frac{1}{2\dt} &[ \|\phiv^{n+1}\|^2 - \|\phiv^{n}\|^2 + \|\phiv^{n+1} - \phiv^n\|^2 ] + \nu \|\nabla \phiv^{n+1}\|^2 + \muo_1 \| \phiv^{n+1}\|^2 \nonumber
	\\ & = 
	- \dt(u_{tt}(t^*), \phiv^{n+1}) 
	- \frac{1}{\dt} (\etav^{n+1}-\etav^n, \phiv^{n+1})
	+\dt (\omega_{tt}(s^*),\phiv^{n+1})
	-(e_{\omega}^n \times v^{n+1}_h, \phiv^{n+1})
	\nonumber
	\\ & \,\,\,\,\,\,\,
		-(\omega^{n+1}\times \etav^{n+1}, \phiv^{n+1})
	 -  2\muo_1(I_H(\phiv^{n+1}) - \phiv^{n+1}, \phiv^{n+1})
	- \muo_1 \|I_H\phiv^{n+1} - \phiv^{n+1}\|^2 
	\nonumber	\\ & \,\,\,\,\,\,\,
	- \muo_1 (I_H \etav^{n+1}, I_H \phiv^{n+1}) , \label{veldiff}
	\end{align}
	and similarly for vorticity, we have
\begin{align}
\frac{1}{2\dt} & [\|\phiw^{n+1}\|^2  - \|\phiw^{n}\|^2 + \|\phiw^{n+1} - \phiw^n\|^2] + \nu \|\nabla \phiw^{n+1}\|^2  + \muo_2\|\phiw^{n+1} \|^2\nonumber
\\ & \,\,\, = - \dt(\omega_{tt}(t^{**}), \phiw^{n+1}) 
- \frac{1}{\dt}(\etaw^{n+1} - \etaw^n, \phiw^{n+1})
+ b^*(\ev^{n+1}, \etaw^{n+1}, \phiw^{n+1})  \nonumber
\\ & \,\,\,\, 
+ b^*(u^{n+1}, \etaw^{n+1}, \phiw^{n+1}) 
+ b^*(\ev^{n+1}, \omega^{n+1}, \phiw^{n+1})
-2\muo_2(I_H(\phiw^{n+1})- \phiw^{n+1}, \phiw^{n+1}) 
\nonumber
\\ & \,\,\,\, 
-\muo_2\|I_H \phiw^{n+1}-\phiw^{n+1}\|^2
-\muo_2(I_H \etaw^{n+1}, I_H \phiw^{n+1}),\label{vortdiffmu}
\end{align}
where in  \eqref{veldiff} we have added and subtracted $\phiv^{n+1}$ to write it in the form found above using
	\begin{align*} 
	\muo_1(I_H \ev^{n+1}, I_H \chih) & = \muo_1(I_H \phiv^{n+1}, I_H \chih) + \muo_1(I_H \etav^{n+1}, I_H \chih) 
	\\ &=  \muo_1(I_H \phiv^{n+1} - \phiv^{n+1} + \phiv^{n+1}, I_H \chih - \phiv^{n+1} + \phiv^{n+1}) + \muo_1(I_H \etav^{n+1}, I_H \chih)  
	\\ & = \muo_1 \|\phiv^{n+1}\|^2 
	+ \muo_1(I_H\phiv^{n+1} - \phiv^{n+1}, \phiv^{n+1})  
	+ \muo_1(\phiv^{n+1}, I_H \chi_h - \phiv^{n+1}) 
	\\ & \,\,\,\, + \muo_1 (I_H \phiv^{n+1} - \phiv^{n+1}, I_H\chi_h - \phiv^{n+1})
	+ \muo_1(I_H \etav^{n+1}, \chih),  
	\end{align*}
	and similarly for \eqref{vortdiffmu}.
		
	Next, we bound the terms on right hand side of difference equations, starting with the velocity difference equation \eqref{veldiff}. The first three right hand side terms are bounded using Cauchy-Schwarz and Young's inequalities, via
	\begin{align*}
	\dt(u_{tt}(t^*), \phiv^{n+1}) &\leq \dt \|u_{tt}\|_{L^\infty(0,\infty;L^2(\Omega))} \|\phiv^{n+1}\|
	\\ &\leq C\dt^2\muo_1^{-1}\|u_{tt}\|^2_{L^\infty(0,\infty;L^2(\Omega))} + \frac{\muo_1}{20}\|\phiv^{n+1}\|^2,
	\end{align*}
	\begin{align*}
	\dt (\omega_{tt}(s^*),\phiv^{n+1})&\leq C\dt \|\omega_{tt}\|_{L^{\infty}(0,\infty;L^{2}(\Omega))} \|u^{n+1}\|_{L^{\infty}} \|\phiv^{n+1}\|
	\\
	&\leq
	C\dt^2 \mu_1^{-1} \|\omega_{tt}\|^2_{L^{\infty}(0,\infty;L^{2}(\Omega))} \|u^{n+1}\|^2_{L^{\infty}} + \frac{\mu_1}{20}\|\phiv^{n+1}\|^2,
	\end{align*}
	\begin{align*}
	\frac{1}{\dt} (\etav^{n+1} - \etav^n, \phiv^{n+1}) & = (\eta_{v,t}(s^*), \phiv^{n+1})
	\\ & \leq \|\eta_{v,t}(s^*)\| \| \phiv^{n+1}\|
	\\ & \leq C \mu_1^{-1}\|\eta_{v,t}(s^*)\|^2  + \frac{\mu_1}{20}\| \phiv^{n+1}\|^2,
	\end{align*}
	where $s^* \in [t^n, t^{n+1}]$.

	For nonlinear terms in \eqref{veldiff}, first we add and subtract $\ew^{n+1}$ in the first component, and $u^{n+1}$ in second component to get
	\begin{align*}
	(\ew^n \times v^{n+1}_h, \phiv^{n+1})
	&=
	((\ew^n-\ew^{n+1}) \times \ev^{n+1}, \phiv^{n+1})
	+
	(\ew^{n+1} \times \ev^{n+1}, \phiv^{n+1}) 
	\\&\,\,\ 
	+ ((\ew^n-\ew^{n+1}) \times u^{n+1}, \phiv^{n+1}) 
	+ (\ew^{n+1} \times u^{n+1}, \phiv^{n+1})
	\\&
	= ((\ew^n-\ew^{n+1}) \times \etav^{n+1}, \phiv^{n+1})
	+
	(\ew^{n+1} \times \etav^{n+1}, \phiv^{n+1})
	\\&\,\,\
	 	+ ((\ew^n-\ew^{n+1}) \times u^{n+1}, \phiv^{n+1}) 
	 + (\ew^{n+1} \times u^{n+1}, \phiv^{n+1}).
	\end{align*} 
	The all resulting terms are bounded by H\"older's and Young's inequalities to obtain
	\begin{align*}
	((\ew^n-\ew^{n+1}) \times \etav^{n+1}, \phiv^{n+1})&\leq C\|\phiw^n-\phiw^{n+1}\| \|\etav^{n+1}\|_{L^{3}} \|\phiv^{n+1}\|_{L^{6}}+ C\|\etaw^n-\etaw^{n+1}\|_{L^{\infty}} \|\etav^{n+1}\| \|\phiv^{n+1}\|
	\\
	&\leq
	C\nu^{-1}\|\phiw^n-\phiw^{n+1}\|^2 \|\etav^{n+1}\|^2_{L^{3}} + \frac{\nu}{8}\|\nabla\phiv^{n+1}\|^2 
	\\
	&\,\,\,\,\,
	+
	C\mu_1^{-1}\|\etaw^n-\etaw^{n+1}\|^2_{L^{\infty}} \|\etav^{n+1}\|^2+\frac{\mu_1}{20} \|\phiv^{n+1}\|^2,
	\end{align*}

	\begin{align*}
	(\ew^{n+1} \times \etav^{n+1}, \phiv^{n+1})&\leq C \|\phiw^{n+1}\| \|\etav^{n+1}\|_{L^{3}} \|\phiv^{n+1}\|_{L^{6}} +C \|\etaw^{n+1}\| \|\etav^{n+1}\|_{L^{\infty}} \|\phiv^{n+1}\|
	\\
	&\leq
	C\nu^{-1} \|\phiw^{n+1}\|^2 \|\etav^{n+1}\|^2_{L^{3}} + \frac{\nu}{8}\|\nabla \phiv^{n+1}\|^2 
	\\
	&\,\,\,\,\,
	+ C\mu_1^{-1}\|\etaw^{n+1}\|^2 \|\etav^{n+1}\|^2_{L^{\infty}} + \frac{\mu_1}{20}\|\phiv^{n+1}\|^2,
	\end{align*}
	
	\begin{align*}
	((\ew^n-\ew^{n+1}) \times u^{n+1}, \phiv^{n+1}) &\leq C\|\phiw^n-\phiw^{n+1}\| \|u^{n+1}\|_{L^{3}} \|\phiv^{n+1}\|_{L^{6}} + C\|\etaw^n-\etaw^{n+1}\| \|u^{n+1}\|_{L^{\infty}} \|\phiv^{n+1}\|
	\\
	&\leq
	C\nu^{-1}\|\phiw^n-\phiw^{n+1}\|^2 \|u^{n+1}\|^2_{L^{3}} + \frac{\nu}{8} \|\nabla \phiv^{n+1}\|^2
	\\
	&\,\,\,\,\, 
	+ C\mu_1^{-1}\|\etaw^n-\etaw^{n+1}\|^2 \|u^{n+1}\|^2_{L^{\infty}} + \frac{\mu_1}{20}\|\phiv^{n+1}\|^2,
	\end{align*}
	
	\begin{align*}
	 (\ew^{n+1} \times u^{n+1}, \phiv^{n+1})& \leq C \|\phiw^{n+1}\| \|u^{n+1}\|_{L^{3}} \|\phiv^{n+1}\|_{L^{6}} + C\|\etaw^{n+1}\| \|u^{n+1}\|_{L^{\infty}} \|\phiv^{n+1}\|
	 \\
	 & \leq C\nu^{-1} \|\phiw^{n+1}\|^2 \|u^{n+1}\|^2_{L^{3}} + \frac{\nu}{8} \|\nabla\phiv^{n+1}\|^2
	 	\\
	 &\,\,\,\,\,
	 + C\mu_1^{-1}\|\etaw^{n+1}\|^2 \|u^{n+1}\|_{L^{\infty}}^2 + \frac{\mu_1}{20}\|\phiv^{n+1}\|^2.
	\end{align*}
	Then, for last nonlinear term, we apply  H\"older's, Poincar\'e's and Young's inequalities and get
	\begin{align*}
	(\omega^{n+1}\times \etav^{n+1}, \phiv^{n+1})&\leq \|\omega^{n+1}\|_{L^{\infty}} \|\etav^{n+1}\| \|\phiv^{n+1}\|
	\\
	&
	\leq C\nu^{-1}\|\omega^{n+1}\|^2_{L^{\infty}} \|\etav^{n+1}\|^2 + \frac{\nu}{8} \|\nabla\phiv^{n+1}\|^2.
	\end{align*} 
	Next, the first interpolation term on the right hand side of \eqref{veldiff} will be bounded with Cauchy-Schwarz inequality and \eqref{interp1} to obtain 
	\begin{align*}
	\muo_1(I_H(\phiv^{n+1}) - \phiv^{n+1}, \phiv^{n+1}) &\leq \muo_1 \|I_H(\phiv^{n+1}) - \phiv^{n+1}\|\| \phiv^{n+1}\|
	\\ & \leq \muo_1 C H \|\nabla \phiv^{n+1}\|\|\phiv^{n+1}\| 
	\\ & \leq C \muo_1 H^2 \|\nabla \phiv^{n+1}\|^2 + \frac{\muo_1}{20}\|\phiv^{n+1}\|^2.
	\end{align*}
	For the second interpolation term, we apply inequality \eqref{interp2}, which yields
	\begin{align*}
	\muo_1\|I_H \phiv^{n+1} - \phiv^{n+1}\|^2 \leq \muo_1 H^2 \|\nabla\phiv^{n+1}\|^2 .
	\end{align*}
	Finally, the last interpolation term will be bounded using Cauchy-Schwarz, \eqref{interp2}, and Young's inequality to get the bound
	\begin{align*}
	\muo_1(I_H\etav^{n+1},I_H \phiv^{n+1}) 
	& \leq C\muo_1\| \etav^{n+1}\|^2 + \frac{\muo_1}{20}\|\phiv^{n+1}\|^2 .
	\end{align*}
	We now move on to the vorticity difference equation, \eqref{vortdiffmu}.  All the linear terms are majorized in a similar manner as in the velocity case, and so we show below the bounds only for the nonlinear terms.  Due to the use of Scott-Vogelius elements, the skew-symmetric form reduces to the usual convective form, so $b^*(u,v,w)=(u\cdot\nabla v,w)$, with $\| \nabla \cdot u \|=0$.  To bound the first nonlinear term on the right hand side of \eqref{vortdiffmu}, we begin by breaking up the velocity error term, then apply H\"older's and Young's inequalities, yielding
		\begin{align*}
	b^*(\ev^{n+1}, \etaw^{n+1}, \phiw^{n+1}) 
	&= 
	(\ev^{n+1}\cdot \nabla \etaw^{n+1}, \phiw^{n+1}) 
	\\
	&
	=
	(\phiv^{n+1}\cdot \nabla \etaw^{n+1}, \phiw^{n+1}) 
	+ (\etav^{n+1}\cdot \nabla \etaw^{n+1}, \phiw^{n+1})
	\\
		&= 
	(\phiv^{n+1}\cdot \nabla\phiw^{n+1}, \etaw^{n+1}) 
	+ (\etav^{n+1}\cdot \nabla\phiw^{n+1}, \etaw^{n+1})
	\\ & \leq 
	C \| \phiv^{n+1}\| \|\nabla \phiw^{n+1}\| \|\etaw^{n+1}\|_{L^\infty} 
	+ C\|\etav^{n+1}\| \|\nabla \phiw^{n+1}\| \|\etaw^{n+1}\|_{L^\infty}
	\\ & \leq
	C\nu^{-1} \|\phiv^{n+1}\|^2 \| \etaw^{n+1}\|_{L^\infty}^2
	+
	\frac{\nu}{10}\|\nabla\phiw^{n+1}\|^2 + C\nu^{-1}\|\etav^{n+1}\|^2 \| \etaw^{n+1}\|_{L^\infty}^2 
	\\ & \,\,\,\,
	+ \frac{\nu}{10} \|\nabla \phiw^{n+1}\|^2.
	\end{align*}
	For the second nonlinear term, we use H\"older's, P\'oincare's and Young's inequalities, which gives
	\begin{align*}
	b^*(u^{n+1}, \etaw^{n+1}, \phiw^{n+1}) &
	=(u^{n+1}\cdot \nabla \etaw^{n+1}, \phiw^{n+1})
	\\
	&
	= (u^{n+1}\cdot \nabla\phiw^{n+1}, \etaw^{n+1})
	 \\ &
	\leq C \|u^{n+1}\|_{L^\infty } \| \nabla \phiw^{n+1}\| \|\etaw^{n+1} \|
	\\ & \leq C \nu^{-1} \|u^{n+1}\|_{L^\infty }^2 \| \etaw^{n+1}\|^2 + \frac{\nu}{10} \| \nabla \phiw^{n+1}\|^2.
	\end{align*}
    For the last nonlinear term, we begin by breaking up the velocity error term, then apply H\"older's and Young's inequalities to get
		\begin{align*}
	b^*(\ev^{n+1}, \omega^{n+1}, \phiw^{n+1}) 
	&=
	(\ev^{n+1}\cdot \nabla \omega^{n+1}, \phiw^{n+1})
	\\
	& = (\phiv^{n+1}\cdot \nabla \omega^{n+1}, \phiw^{n+1}) + (\etav^{n+1}\cdot \nabla \omega^{n+1}, \phiw^{n+1})\\
	 & = (\phiv^{n+1}\cdot \nabla \phiw^{n+1},\omega^{n+1}) + (\etav^{n+1}\cdot \nabla\phiw^{n+1},\omega^{n+1})
	\\ & \leq  
	C\|\phiv^{n+1}\| \|\nabla \phiw^{n+1}\| \|\omega^{n+1}\|_{L^\infty}
	+
	C \|\etav^{n+1}\| \|\nabla\phiw^{n+1}\| \|\omega^{n+1}\|_{L^\infty}
	\\ & \leq 
	C \nu^{-1}\|\phiv^{n+1}\|^2 \| \omega^{n+1}\|_{L^\infty}^2 
	+ \frac{\nu}{10} \|\nabla \phiw^{n+1}\|^2  
	+C\nu^{-1} \|\etav^{n+1}\|^2 \| \omega^{n+1}\|_{L^\infty}^2  
	\\ & \,\,\,\, + \frac{\nu}{10}\|\nabla \phiw^{n+1}\|^2.
	\end{align*}
	
	Replacing the right hand sides of \eqref{veldiff} and \eqref{vortdiffmu} with the computed bounds and dropping nonnegative terms with $\|\phiv^{n+1}-\phiv^{n}\|^2 $ yields the bound
\begin{align*}
\frac{1}{2\dt}& \left( \|\phiv^{n+1}\|^2 + \|\phiw^{n+1}\|^2 - \|\phiv^n\|^2 - \|\phiw^n\|^2   \right) 
+
 \left(\frac{1}{2\dt}-C\nu^{-1} (\|\etav^{n+1}\|^2_{L^3} -\|u^{n+1}\|^2_{L^3} )\right) \|\phiw^{n+1}-\phiw^{n}\|^2
\\ &\ \ \ 
+ 
\frac{\nu}{4} \|\nabla \phiv^{n+1}\|^2
+
\left( \frac{\nu}{4} - C \muo_1 H^2  \right)\|\nabla \phiv^{n+1}\|^2
+
\frac{\nu}{4} \|\nabla \phiw^{n+1}\|^2 
+
\left( \frac{\nu}{4} - C \muo_2 H^2  \right)\|\nabla \phiw^{n+1}\|^2 
\\ &\ \ \ 
+
\frac{\muo_1}{4} \|\phiv^{n+1}\|^2 
+
\left( \frac{\muo_1}{4}-C\nu^{-1}(\|\omega^{n+1}\|^2_{L^{\infty}}+\|\etaw^{n+1}\|^2_{L^{\infty}}) \right) \|\phiv^{n+1}\|^2 
\\ &\ \ \ 
+
\frac{\muo_2}{4} \|\phiw^{n+1}\|^2 
+ 
\left( \frac{\muo_2}{4}-C\nu^{-1}(\|u^{n+1}\|^2_{L^{3}}+\|\etav^{n+1}\|^2_{L^{3}}) \right) \|\phiw^{n+1}\|^2 
\\ & 
\leq 
C \dt^2 
\left(
\muo_1^{-1}\|u_{tt}\|^2_{L^\infty(0,\infty ; L^2)} +\muo_1^{-1}\|\omega_{tt}\|^2_{L^\infty(0,\infty ; L^2)} \|u^{n+1}\|^2_{L^{\infty}} + \mu_2^{-1}\|\omega_{tt}\|^2_{L^\infty(0,\infty; L^2)} 
\right)
\\ & \,\,\,\, 
+
C \mu_1^{-1}
\big( 
\|\eta_{v,t}\|_{L^\infty(0,\infty;L^2)}^2 
+
\|\etaw^{n+1}-\etaw^{n}\|^2_{L^{\infty}} \|\etav^{n+1}\|^2
+
\|\etaw^{n+1}\|^2 \|\etav^{n+1}\|^2_{L^{\infty}}
+
\|\etaw^{n+1}-\etaw^{n}\|^2 \|u^{n+1}\|^2_{L^{\infty}}
\\ & \,\,\,\, 
+
\|\etaw^{n+1}\|^2 \|u^{n+1}\|^2_{L^{\infty}}
\big)
+
 C\mu_2^{-1}
 \|\eta_{w,t}\|_{L^\infty(0,\infty;L^2)}^2
+
 C \nu^{-1}
\big(
\|\omega^{n+1}\|^2_{L^{\infty}} \|\etav\|^2
+
 \|\etav^{n+1}\|^2\|\etaw^{n+1}\|_{L^\infty}^2 
 \\ & \,\,\,\, 
+
 \|\etav^{n+1}\|^2 \|\omega^{n+1}\|_{L^\infty}^2
+
\|\etaw^{n+1}\|^2\|u^{n+1}\|_{L^\infty}^2  
\big)
+ C \mu_1 \|\etav^{n+1}\|^2
+ C \mu_2 \|\etaw^{n+1}\|^2.  
\end{align*}
Using the assumptions on $H$ and the nudging parameters, the time step restriction, and smoothness of the true solution, this reduces to 
\begin{align*}
\frac{1}{2\dt}& \left( \|\phiv^{n+1}\|^2 + \|\phiw^{n+1}\|^2 - \|\phiv^n\|^2 - \|\phiw^n\|^2   \right) 
+ 
\frac{\nu}{4} \|\nabla \phiv^{n+1}\|^2
+
\frac{\nu}{4} \|\nabla \phiw^{n+1}\|^2 
+
\frac{\muo_1}{4} \|\phiv^{n+1}\|^2 
+
\frac{\muo_2}{4} \|\phiw^{n+1}\|^2 
\\ & 
\leq 
C \dt^2 
\left(
\muo_1^{-1} +\muo_1^{-1} + \mu_2^{-1} \right)
+
 C\mu_2^{-1}
 \|\eta_{w,t}\|_{L^\infty(0,\infty;L^2)}^2
\\ & \ 
+
C \mu_1^{-1}
\big( 
\|\eta_{v,t}\|_{L^\infty(0,\infty;L^2)}^2 
+
\|\etaw^{n+1}-\etaw^{n}\|^2_{L^{\infty}} \|\etav^{n+1}\|^2
+
\|\etaw^{n+1}\|^2 \|\etav^{n+1}\|^2_{L^{\infty}}
+
\|\etaw^{n+1}-\etaw^{n}\|^2 
+
\|\etaw^{n+1}\|^2 
\big)
\\ & \ 
 +C \nu^{-1}
\big(
 \|\etav\|^2
+
 \|\etav^{n+1}\|^2\|\etaw^{n+1}\|_{L^\infty}^2 
+
 \|\etav^{n+1}\|^2 
+
\|\etaw^{n+1}\|^2
\big)
+ C \mu_1 \|\etav^{n+1}\|^2
+ C \mu_2 \|\etaw^{n+1}\|^2.  
\end{align*}

Now define
	\begin{align*}
	\lambda_1 : =& \frac{\mu_1}{4} + \frac{\nu C_P^{-2}}{4} ,\\ 
   \lambda_2 : =& \frac{\mu_2}{4} +  \frac{\nu C_P^{-2}}{4} .
	\end{align*} 
Using this in the inequality after applying Poincare's inequality and multiplying each side by $2\dt$, we get  
	\begin{align*}
	(1 + \dt\lambda_1  )\|\phiv^{n+1}\|^2 +(1 + \dt\lambda_2 ) \|\phiw^{n+1}\|^2 & \leq C \dt \left(\mu_1^{-1} \dt^2 + \nu^{-1} \dt^2 + \nu^{-1} h^{2k+2} + \mu_1 h^{2k+2} + \mu_2 h^{2k+2}\right)
	\\ & \,\,\,\, 
	+ \|\phiv^{n}\|^2 + \|\phiw^{n}\|^2.
	\end{align*}
	Then, with 
	$R:=\left(\mu_1^{-1} \dt^2 + \nu^{-1} \dt^2 + \nu^{-1} h^{2k+2} + \mu_1 h^{2k+2} + \mu_2 h^{2k+2}\right)$ and
	$ \lambda : = \min \left\{ \lambda_1, \lambda_2  \right\}$, we obtain the bound		
	\begin{align*}
	(1 +\lambda\dt )\left(\|\phiv^{n+1}\|^2 +\|\phiw^{n+1}\|^2\right)  \leq C \dt R + \|\phiv^{n}\|^2 + \|\phiw^{n}\|^2.
	\end{align*}		
	By Lemma \ref{geoseries}, this implies
	\begin{align*}
	\|\phiv^{n+1}\|^2 + \|\phiw^{n+1}\|^2 & \leq C \lambda^{-1} R + (1 + \lambda \dt )^{-(n+1)}(\|\phiv^{0}\|^2 + \|\phiw^{0}\|^2).
	\end{align*}
	Lastly, applying triangle inequality completes the proof.
\end{proof}

\begin{theorem}[Long-time $L^2$ accuracy of Algorithm \ref{VVNLbe} with $\mu_1>0, \mu_2 =0$ ]
	Let true solution $u \in L^\infty(0, \infty ; H^{k+2}(\Omega))$, $ p \in L^\infty(0, \infty ; H^{k}(\Omega))$ where $k\geq 1$ and $u_{t}, u_{tt}, \in L^\infty (0, \infty ; H^1)$. Then, assume that time step $\dt$ is sufficiently small, $\mu_2=0$, and that $\mu_1$ satisfies
	$$\max \left\{1 , C\nu^{-1}(\|u^{n+1}\|_{L^{\infty}}+\|\etav^{n+1}\|_{L^{\infty}}) , C\nu^{-1}(\|\omega^{n+1}\|_{L^{\infty}}+\|\etaw^{n+1}\|_{L^{\infty}} ) \right\} \leq \mu_1 \leq \frac{C\nu}{H^2},$$
    where $H$ is chosen so that this inequality holds. Then for any time $t^n$, $n = 0, 1, 2,...$, solutions of of Algorithm \ref{VVNLbe} using Scott-Vogelius element satisfy 
	\begin{align}\label{l2accboth}
	\|v_h^n-u^n\|^2 + \|\omega_h^n - \rot u^n\|^2 & \leq
	(1 + \lambda \dt )^{-n}(\|v_h^0-u^0\|^2 + \|\omega_h^0 - \rot u^0\|^2)
	+ C\lambda^{-1} R,
	\end{align}
	where  
	\[R:=\left(\mu_1^{-1} \dt^2 + \nu^{-1} \dt^2 + \nu^{-1} h^{2k+2} + \mu_1 h^{2k+2} \right),   \]
	and $\lambda= \frac{C_P^{-2} \nu}{4} $ with $C$ independent of $\dt$, $h$ and $H$ . 
\end{theorem}


\begin{remark}
Algorithm \ref{VVNLbe} converges to the true solutions up to optimal discretization error in both cases $\mu_1, \mu_2 >0$ and $\mu_1>0, \mu_2=0$. The key difference between two cases is that when $\mu_2=0$, the convergence in time to reach optimal accuracy is much slower since $\lambda$ does not scale with the nudging parameters.  This phenomena is illustrated in our numerical tests.
\end{remark}

\begin{proof}
We follow the same steps with the proof of Theorem \ref{nlthmL2}.  The difference equation for velocity is already the same with (\ref{veldiff}), and just two nonlinear terms in the velocity difference equation are bounded with differently in this case.  By H\"older, Poincar\'e and Young's inequalities, we get the bounds

	\begin{align*}
(\ew^{n+1} \times \etav^{n+1}, \phiv^{n+1})&\leq C \|\phiw^{n+1}\| \|\etav^{n+1}\|_{L^{3}} \|\phiv^{n+1}\|_{L^{6}} +C \|\etaw^{n+1}\| \|\etav^{n+1}\|_{L^{\infty}} \|\phiv^{n+1}\|
\\
&\leq
C \muo_1^{-1}\|\nabla\phiw^{n+1}\|^2\|\etav^{n+1}\|^2_{L^\infty} + \frac{\muo_1}{16}\| \phiv^{n+1}\|^2
\\
&\,\,\,\,\,
+ C\mu_1^{-1}\|\etaw^{n+1}\|^2 \|\etav^{n+1}\|^2_{L^{\infty}} + \frac{\mu_1}{20}\|\phiv^{n+1}\|^2,
\end{align*}

\begin{align*}
(\ew^{n+1} \times u^{n+1}, \phiv^{n+1})& \leq C \|\phiw^{n+1}\| \|u^{n+1}\|_{L^{3}} \|\phiv^{n+1}\|_{L^{6}} + C\|\etaw^{n+1}\| \|u^{n+1}\|_{L^{\infty}} \|\phiv^{n+1}\|
\\
& \leq  C \muo_1^{-1}\|\nabla\phiw^{n+1}\|^2\|u^{n+1}\|_{L^\infty}^2 + \frac{\muo_1}{16}\| \phiv^{n+1}\|^2
\\
&\,\,\,\,\,
+ C\mu_1^{-1}\|\etaw^{n+1}\|^2 \|u^{n+1}\|_{L^{\infty}}^2 + \frac{\mu_1}{20}\|\phiv^{n+1}\|^2.
\end{align*}
All terms on the right hand side of vorticity difference equation for Theorem \ref{nlthmL2} are bounded identically.  Proceeding as in the previous proof, we arrive at
 \begin{align*}
 \frac{1}{2\dt}& \left( \|\phiv^{n+1}\|^2 + \|\phiw^{n+1}\|^2 - \|\phiv^n\|^2 - \|\phiw^n\|^2   \right) 
 +
 \left(\frac{1}{2\dt}-C\nu^{-1} (\|\etav^{n+1}\|^2_{L^3} -\|u^{n+1}\|^2_{L^3} )\right) \|\phiw^{n+1}-\phiw^{n}\|^2
 \\ &\ \ \ 
 +
 \frac{\muo_1}{4} \|\phiv^{n+1}\|^2 
 +
 \left( \frac{\muo_1}{4}-C\nu^{-1}(\|\omega^{n+1}\|^2_{L^{\infty}}+\|\etaw^{n+1}\|^2_{L^{\infty}}) \right) \|\phiv^{n+1}\|^2
 + 
\frac{\nu}{4} \|\nabla \phiv^{n+1}\|^2
 \\ & \,\,\,\, 
 +
\left( \frac{\nu}{4} - C \muo_1 H^2  \right)\|\nabla \phiv^{n+1}\|^2
+
\frac{\nu}{4} \|\nabla \phiw^{n+1}\|^2 
+
\left( \frac{\nu}{4} -C\mu_1^{-1}(\|u^{n+1}\|_{L^{\infty}}+\|\etav^{n+1}\|_{L^{\infty}})   \right)\|\nabla \phiw^{n+1}\|^2 
 \\ & 
 \leq 
 C \dt^2 
 \left(
 \muo_1^{-1}\|u_{tt}\|^2_{L^\infty(0,\infty ; L^2)} +\muo_1^{-1}\|\omega_{tt}\|^2_{L^\infty(0,\infty ; L^2)} \|u^{n+1}\|^2_{L^{\infty}} + \mu_2^{-1}\|\omega_{tt}\|^2_{L^\infty(0,\infty; L^2)} 
 \right)
 \\ & \,\,\,\, 
 +
 C \mu_1^{-1}
 \big( 
 \|\eta_{v,t}\|_{L^\infty(0,\infty;L^2)}^2 
 +
 \|\etaw^{n+1}-\etaw^{n}\|^2_{L^{\infty}} \|\etav^{n+1}\|^2
 +
 \|\etaw^{n+1}\|^2 \|\etav^{n+1}\|^2_{L^{\infty}}
 +
 \|\etaw^{n+1}-\etaw^{n}\|^2 \|u^{n+1}\|^2_{L^{\infty}}
 \\ & \,\,\,\, 
 +
 \|\etaw^{n+1}\|^2 \|u^{n+1}\|^2_{L^{\infty}}
 \big)
 +
 C \nu^{-1}
 \big(
 \|\omega^{n+1}\|^2_{L^{\infty}} \|\etav\|^2
 +
 \|\etav^{n+1}\|^2\|\etaw^{n+1}\|_{L^\infty}^2 
 +
 \|\etav^{n+1}\|^2 \|\omega^{n+1}\|_{L^\infty}^2
  \\ & \,\,\,\, 
 +
 \|\etaw^{n+1}\|^2\|u^{n+1}\|_{L^\infty}^2  
 \big)
 + C \mu_1 \|\etav^{n+1}\|^2.
 \end{align*}
Provided $\dt$ is sufficiently small and the restriction
$$\max \left\{1, C\nu^{-1}(\|u^{n+1}\|_{L^{\infty}}+\|\etav^{n+1}\|_{L^{\infty}}) , C\nu^{-1}(\|\omega^{n+1}\|_{L^{\infty}}+\|\etaw^{n+1}\|_{L^{\infty}} ) \right\} \leq \mu_1 \leq \frac{C\nu}{H^2},$$
holds, then applying Poincar\'e inequality to the terms on left hand side and using
\begin{align*}
\lambda_1 : =& \frac{\mu_1}{4}+\frac{C_P^{-2} \nu}{4},\\ 
\lambda_2 : =& \frac{C_P^{-2}\nu}{4},
\end{align*} 
and assumptions on the true solution, we obtain
	\begin{align*}
(1 + \dt\lambda_1  )\|\phiv^{n+1}\|^2 +(1 + \dt\lambda_2 ) \|\phiw^{n+1}\|^2 & \leq C \dt \left(\mu_1^{-1} \dt^2 + \nu^{-1} \dt^2 + \nu^{-1} h^{2k+2} + \mu_1 h^{2k+2} \right)
\\ & \,\,\,\, 
+ \|\phiv^{n}\|^2 + \|\phiw^{n}\|^2.
\end{align*}
From here, the proof is finished in the same way as the previous theorem.

\end{proof}


\subsection{Second order temporal discretization}

We now present results for a second order analogue of the first order algorithm studied above.  
\begin{algorithm} \label{VVNLbdf2}
	Find $(v_h^{n+1}, w_h^{n+1}, q_h^{n+1}) \in (X_h, W_h, Q_h)$ for $n = 0,1,2,...$, satisfying 
	\begin{align}
	\frac{1}{2 \Delta t} \left( 3v_h^{n+1} - 4v_h^n + v_h^{n-1},\chi_h \right) + ( (2w_h^{n} - w_h^{n-1}) \times v_h^{n+1}, \chi_h) - (P_h^{n+1},\nabla \cdot \chi_h) + \nu (\nabla & v_h^{n+1},\nabla \chi_h)   \nonumber \\ + \muo_1 (I_H(v_h^{n+1} - u^{n+1}),I_H \chi_h)&=  (f^{n+1},\chi_h), \label{nlda5a} \\
	(\nabla \cdot v_h^{n+1},r_h)  &= 0, \label{nlda6a}
	\\ \frac{1}{2\Delta t} \left( 3w_h^{n+1} - 4w_h^n - v_h^{n-1},\psi_h \right) + (v_h^{n+1}\cdot \nabla w_h^{n+1},\psi_h)  + \nu (\nabla w_h^{n+1},\nabla \psi_h)  \nonumber\\
		+ \muo_2 (I_H(w_h^{n+1} - \rot u^{n+1}),I_H(\psi_h)) &=  (\rot f^{n+1},\psi_h), \label{nlda7a}
	\end{align}
	for all $(\chi_h, \psi_h, r_h) \in (X_h,W_h,Q_h)$, with $v^0\in X$ and $I_H(u^{n+1})$, $I_H(\rot u^{n+1})$ given.

\end{algorithm}

Stability and convergence results follow in the same manner as the first order scheme results above, using G-stability theory as in \cite{AKR17,articleAkbas,LRZ19}.

%

\begin{theorem}[Long-time stability and accuracy of Algorithm \ref{VVNLbdf2} with $\muo_1>0$ and $\muo_2>0$] \label{nlthmbdf2L2}
For any time step $\dt>0$, and any time $t^n$, $n = 0, 1, 2,...$, we have that solutions of Algorithm \ref{VVNLbdf2} satisfy
	\[
	\| v_h^n \| + \| w_h^n \| + \| \nabla v_h^n \| + \| \nabla w_h^n \| \le C,
	\]
	with $C$ independent of $n$, $\dt,\ h,\ H$.
	
	Furthermore, if we suppose the true solution $u \in L^\infty(0, \infty ; H^{k+2}(\Omega))$, $ p \in L^\infty(0, \infty ; H^{k}(\Omega))$ where $k\geq 1$ and $u_{t}, u_{ttt}, \in L^\infty (0, \infty ; H^1)$, that time step $\dt$ is sufficiently small, Scott-Vogelius elements are used, and that $\mu_1$ and $\muo_2$ satisfy
	$C(u) \leq \mu_1,\mu_2 \leq \frac{C\nu}{H^2}$, we have the bound
	\begin{align*}
	\|v_h^n-u^n\|^2 & + \|\omega_h^n - \rot u^n\|^2  \leq \\
	& (1 + \lambda \dt )^{-n}(\|v_h^0-u^0\|^2 + \|\omega_h^0 - \rot u^0\|^2 + \|v_h^1-u^1\|^2 + \|\omega_h^1 - \rot u^1\|^2)
	+ C \lambda^{-1} R,
	\end{align*}
	where
	\[R:=\left(\mu_1^{-1} \dt^4 + \mu_2^{-1} \dt^4 + \nu^{-1} h^{2k+2} + \mu_1 h^{2k+2} + \mu_2 h^{2k+2}\right),   \]
	and
	$\lambda=\min \left\{ \frac{\mu_1}{4} + \frac{\nu C_P^{-2}}{4} , \frac{\mu_2}{4} +  \frac{\nu C_P^{-2}}{4} \right\}$ with $C$ independent of $\dt$, $h$ and $H$. 
\end{theorem}

\section{Numerical Experiments}
In this section, we illustrate the above theory with two numerical tests, both using Algorithm \ref{VVNLbdf2}.  Our first test is for convergence rates on a problem
with analytical solution, and the second test is for flow past a flat plate.  For both tests, we report results only for $(P_2,P_1)$ Taylor-Hood elements for velocity and pressure, and $P_2$ for vorticity; however we also tried Scott-Vogelius elements on barycenter refined meshes that produced similar numbers of degrees of freedom, and results were very similar to those of Taylor-Hood.  The coarse velocity and vorticity spaces $X_H$ and $W_H$ are defined to be piecewise constants on the same mesh used for the computations.  The interpolation operator $I_H$ was taken to be the $L^2$ projection operator onto $X_H$ (or $W_H$), which is known to satisfy \eqref{interp1}-\eqref{interp2} \cite{fem:book:ern:guermond}.

\subsection{Experiment 1:  convergence rate test}
For our first test, we investigate the theory above for Algorithm \ref{VVNLbdf2}.  Here we use the analytic solution
\[ 
u = \begin{bmatrix}
 \cos(\pi(y-t))\\\sin(\pi(x+t))
\end{bmatrix}, \quad \quad p = (1+t^2)\sin(x+y),  
\]
on the unit square domain $\Omega=(0,1)^2$ with kinematic viscosity $\nu=1.0$, and use and the NSE to determine $f$ and boundary conditions.  We take the final time $T=1$, and choose initials conditions for  Algorithm \ref{VVNLbdf2}'s velocity and vorticity to be 0.  For the discretization, $(P_2,P_1)$ Taylor-Hood elements are used for velocity and pressure, $P_2$ for vorticity, and a time step size of $\Delta t=0.001$.  From Section \ref{analysisDA}, we expect third order spatial convergence rate in the $L^2$ norm for large enough times.   Results are presented below for two cases, $\mu_2>0$ and $\mu_2=0$.

\subsubsection{Results for $\mu_1>0$ and $\mu_2>0$}

To test this case, we first calculated spatial convergence rates at the final time $T=1$ with the $L^2$ error, using successively refined uniform meshes and $\mu_1=\mu_2=100$.  Errors and rates are shown in table \ref{rates1}, and show clear third order spatial convergence of both velocity and vorticity.  Deterioration of the rates for the smallest $h$ is expected since the time step $\Delta t$ is fixed while the spatial mesh width decreases.

\begin{table}[H]
	\centering
				\begin{tabular}{|l|c|c|c|c|}
				\hline
				h& $\|e_v(T) \|$ & rate & $\|e_w(T)\|$ & rate\\
				\hline
				1/4 &2.62008e-03&   -  &7.70647e-03&   -  \\
				1/8  &3.20467e-04&3.0314&9.68456e-04&2.9923\\
				1/16 &3.97307e-05&3.0146&1.20888e-04&3.0041\\
				1/32 &4.94529e-06&3.0061&1.50809e-05&3.0029\\
				1/64 &6.19332e-07&2.9973&1.99325e-06&2.9195\\
				1/128&8.13141e-08&2.9247&3.15236e-07&2.5855\\
				\hline
			\end{tabular}
	\caption{Shown above are $L^2$ velocity and vorticity errors and convergence rates on varying mesh widths, at the final time $T=1$, using Algorithm \ref{VVNLbdf2} with $\muo_1=\muo_2=100$. }\label{rates1}
\end{table}

We next consider convergence to the true solution exponentially in time (up to discretization error).  Here we take $h=1/32$, and compute solutions using $\mu_1=\mu_2=\mu$, with $\mu=1,\ 10,\ 100,\ 1000$.  Results are shown in 
figure \ref{covplot1}, as $L^2$ error versus time for velocity and vorticity.  We observe exponential convergence in time of both velocity
and vorticity, up to about $10^{-5}$, which is consistent with the choices of $h$ and $\Delta t$ and the accuracy of the method.  We note that
as $\mu$ is increased, convergence is faster in time, which is consistent with our theory for the case of $\mu_1>0$ and $\mu_2>0$.

\begin{figure}[H]
	\centering
	\includegraphics[width = .45\textwidth, height=.30\textwidth]{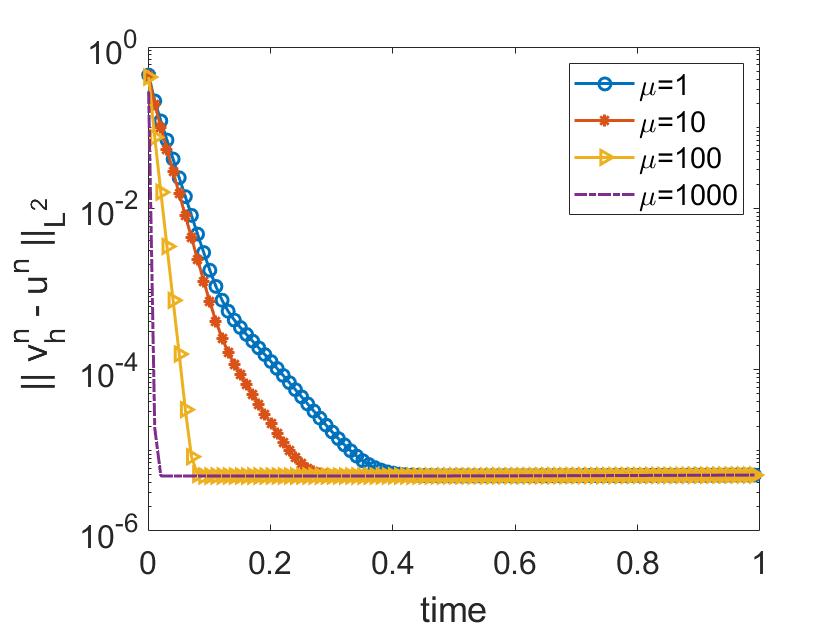}	
	\includegraphics[width = .45\textwidth, height=.30\textwidth]{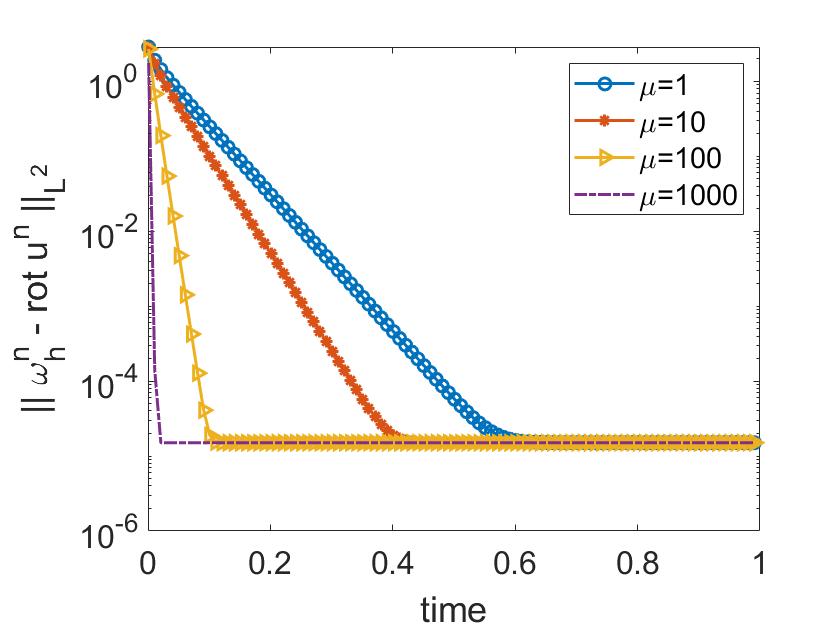}
	\caption{Shown above are $L^2$ velocity and vorticity errors for Algorithm \ref{VVNLbdf2} with $\mu_1=\mu_2=\muo$, with varying $\mu>0$.}\label{covplot1}
\end{figure}

\subsubsection{Results for $\mu_1>0$ and $\mu_2=0$}

We now consider the same tests as above, but with $\mu_2=0$.  This is an important case, since it is not always practical to obtain accurate vorticity measurement data.
Just as in the first case, we first calculated spatial convergence rates at the final time $T=1$ for the $L^2$ error, on the same successively refined uniform meshes, but now with $\mu_1=100$ and $\mu_2=0$.  Errors and 
rates are shown in table \ref{rates1}, and show clear third order spatial convergence of both velocity and vorticity.  Deterioration of the rates for the smallest $h$ is expected since
the time step $\Delta t$ was fixed at 0.001, although the vorticity errors are slightly worse than for the case of $\mu_2=100$ shown in table \ref{rates1}, and the deterioration of the rates occurs a bit earlier.  Hence we observe essentially the same velocity errors and rates compared to the case of $\mu_2=100$, and slightly worse vorticity error but still with optimal $L^2$ accuracy.

\begin{table}[H]
	\centering
		\begin{tabular}{|l|c|c|c|c|}
			\hline
			h& $\|e_v(T)\|$ & rate & $\|e_w(T)\|$ & rate \\
			\hline
			1/4  &2.62003e-03&   -  &7.79431e-03&  -  \\
			1/8  &3.20466e-04&3.0313&9.70492e-04&3.0056\\
			1/16 &3.97175e-05&3.0123&1.20897e-04&3.0049\\
			1/32 &4.94501e-06&3.0057&1.50883e-05&3.0023\\
			1/64 &6.17406e-07&3.0017&2.08215e-06&2.8573\\
			1/128&8.11244e-08&2.9280&9.37122e-07&1.1518\\
			\hline
		\end{tabular}

	\caption{Shown above are $L^2$ velocity and vorticity errors and convergence rates on varying mesh widths, at the final time $T=1$, using Algorithm \ref{VVNLbdf2} with $\muo_1=100$ and $\muo_2=100$. }\label{rates}
\end{table}

To test exponential convergence in time for the case of $\mu_2=0$, we again take $h=1/32$, and compute solutions using $\mu_1=1,\ 10,\ 100,\ 1000$.  Results are shown in 
figure \ref{covplot}, as $L^2$ error versus time for velocity and vorticity.  Although we do observe exponential convergence in time of both velocity
and vorticity, up to about $10^{-5}$ which is the same accuracy reached when $\mu_2=100$ above.
An important difference here compared to when $\mu_2=100$ is that the convergence of vorticity to the true solution is independent of $\mu_1$, and the convergence of
velocity is slower for larger choices of $\mu_1$.  This reduced dependence of the convergence on the nudging parameters when $\mu_2=0$ is consistent with our theory.  Hence without vorticity nudging, long-time optimal accuracy is still achieved, but it takes longer in time to get there.

\begin{figure}[H]
	\centering
	\includegraphics[width = .45\textwidth, height=.30\textwidth]{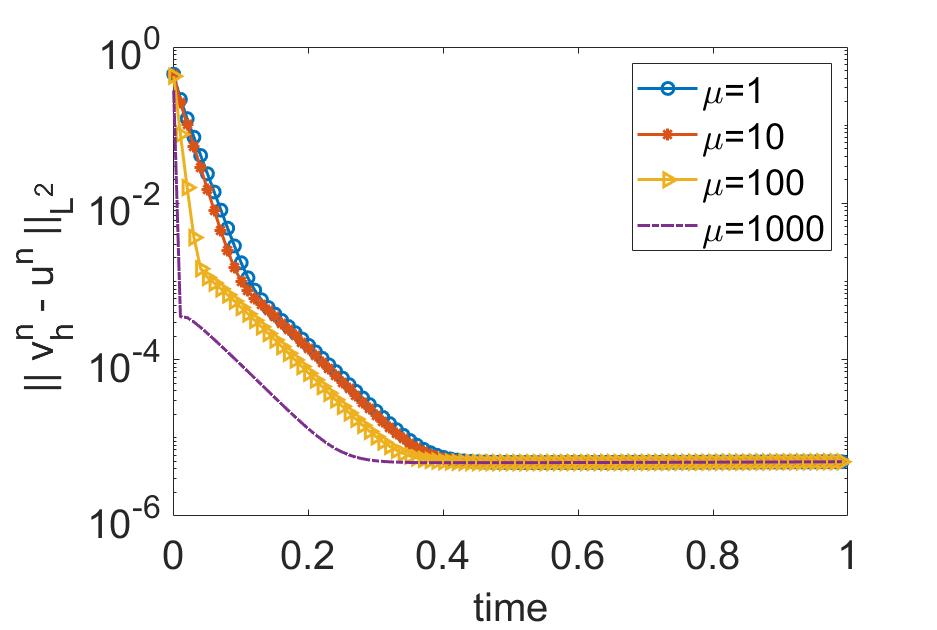}	
	\includegraphics[width = .45\textwidth, height=.30\textwidth]{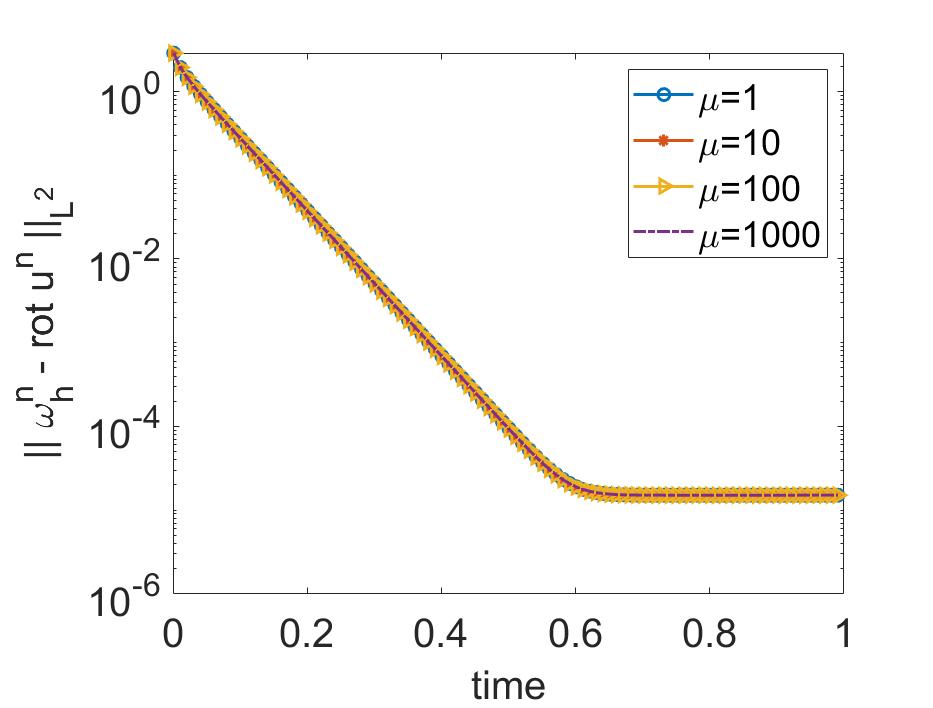}
	\caption{Shown above are $L^2$ velocity and vorticity errors (from left to right) for Algorithm \ref{VVNLbdf2} with varying $\mu_1$ and $\mu_2=0$. }\label{covplot}
\end{figure}

\subsection{Experiment 2: Flow past a normal flat plate}
{\color{black}
To test Algorithm \ref{VVNLbdf2} on a more practical problem, we consider flow past a flat plate with $Re=50$.  The domain of this problem is $[-7,20]\times[-10,10]$ rectangular channel with a $0.125\times1$ plate fixed ten units into the channel from the left, vertically centered. The inflow velocity is  $u_{in}=\langle0,1\rangle$, no-slip velocity and the corresponding natural vorticity boundary condition from \cite{Olshanskii2015NaturalVB} are used on the walls and plate, and homogeneous Neumann conditions are enforced weakly at the outflow.  A setup for the domain is shown in figure \ref{fig:plate_setup}.}  There is no external forcing applied, $f=0$. The viscosity is taken to be $\nu=1/50$ which is inversely proportional to $Re$, based on the height of the plate.  The end time for the test is $T=80$.  A DNS was run until for 160 time units (from t=-80 to t=80), and for $t>0$ measurement data for the VV-DA simulation was sampled from the DNS.

\begin{figure}[ht]
	\centering
\includegraphics[width=0.4\textwidth]{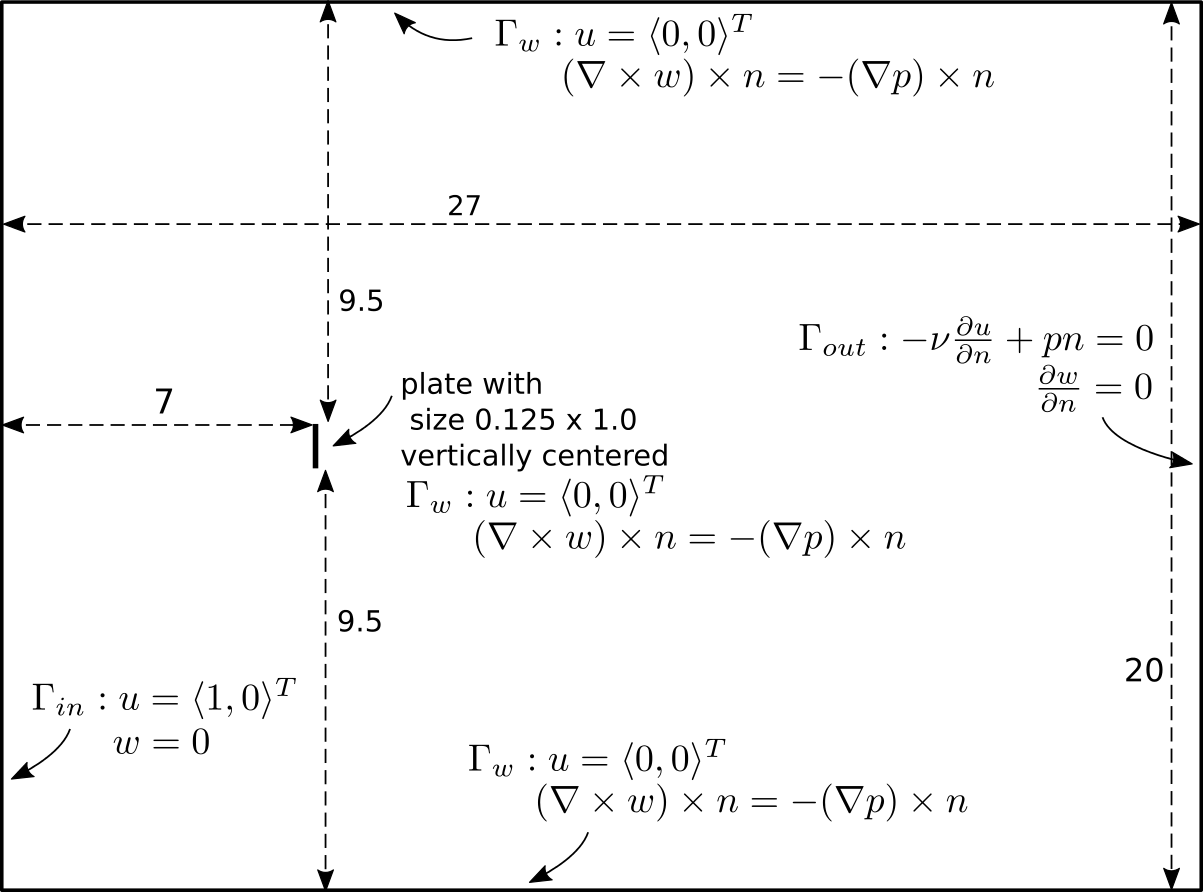}
 \caption{Setup for the flow past a normal flat plate.}
 \label{fig:plate_setup}
\end{figure}

We computed solutions using a Delaunay generated triangular meshes that provided $27,373$ total degree of freedom with $(P^2, P^1, P^2)$ velocity-pressure-vorticity elements, and time step $\dt=0.02$.  We first compared convergence in time to the DNS solution, for two cases: $\mu_1=\mu_2>0$ and $\mu_1>-0,\ \mu_2=0$.  Plots of $L^2$ velocity and vorticity error for both of these cases are shown in figure \ref{errorplot}, with varying nudging parameters.  There is a clear advantage seen in the plots for the simulations with $\mu_2>0$: when vorticity is nudged in addition to velocity, convergence to the true solution is much faster in time.  The convergence when $\mu_2=0$ appears to still be occurring, but is much slower and even by $t=80$ the $L^2$ vorticity error is barely smaller than $O(1)$.  We note that just like in the analytical test problem, when $\mu_2=0$ the vorticity convergence in time is independent of $\mu_1$.

\begin{figure}[H]
	\centering
			\includegraphics[width = .45\textwidth, height=.35\textwidth]{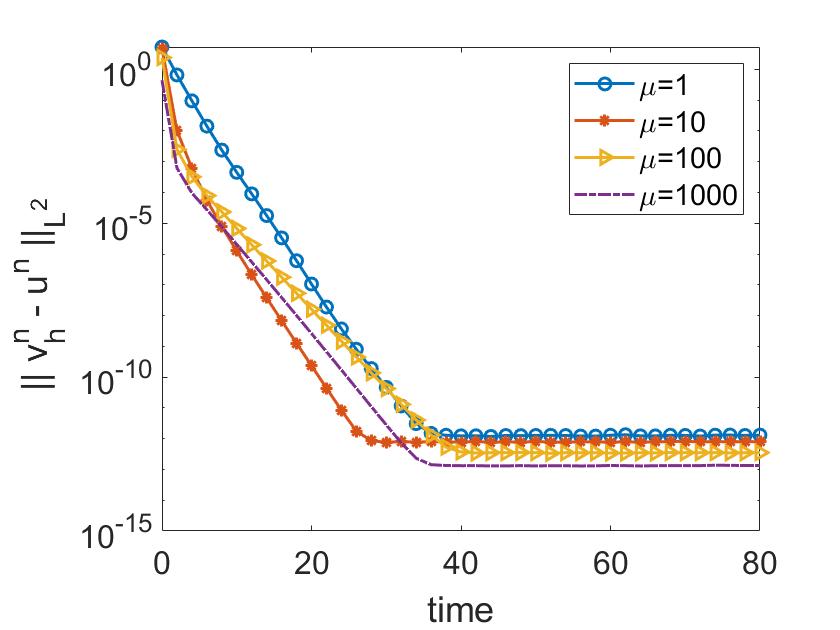}
	\includegraphics[width = .45\textwidth, height=.35\textwidth]{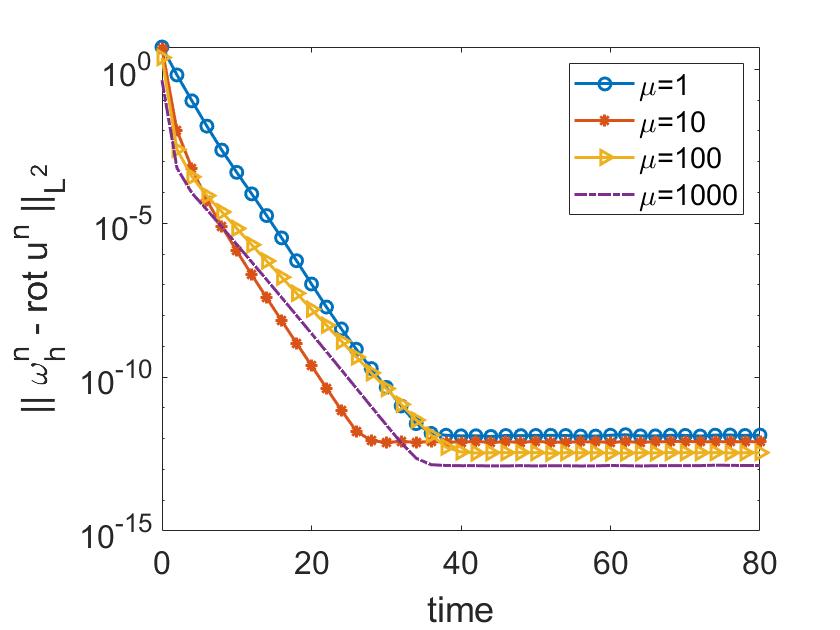}
	\\
	\includegraphics[width = .45\textwidth, height=.35\textwidth]{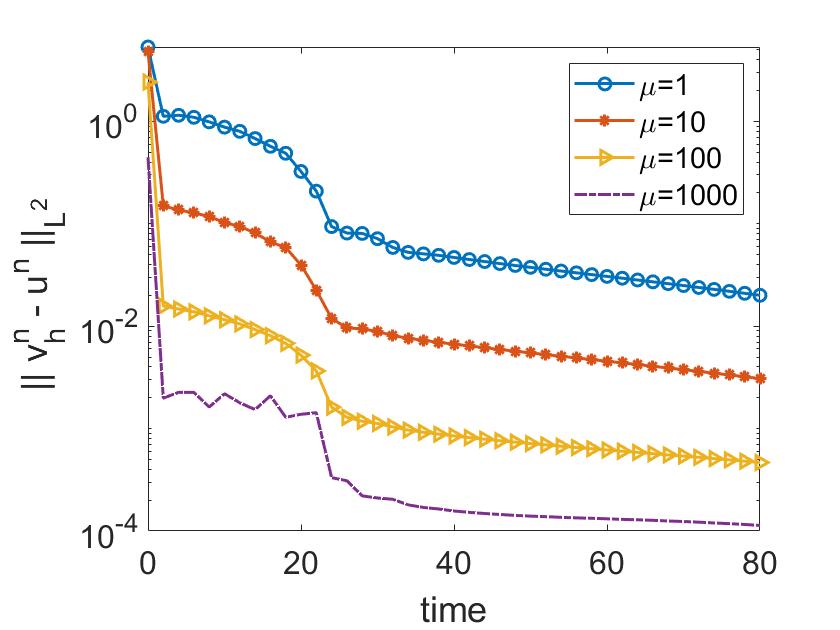}
	\includegraphics[width = .45\textwidth, height=.35\textwidth]{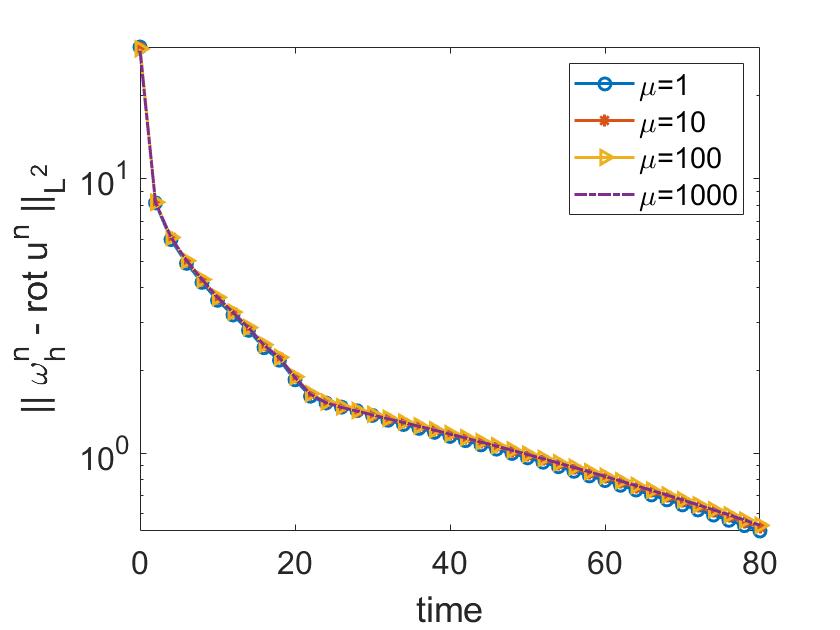}
	\caption{$L^2$ velocity and vorticity errors (from left to right) for Algorithm \ref{VVNLbdf2} with $\mu_1=\mu_2=\muo >0$(top) and $\mu_1=\mu>0, \mu_2=0$ (bottom) }\label{errorplot}
\end{figure}

To further illustrate the convergence of the DNS, we show contour plots of the DNS solution, the VVDA solution, and their difference, in figures \ref{both_vel}-\ref{justvel_vort}.  For these simulations, we used $\mu_1=\mu_2=10$ in figures \ref{both_vel}-\ref{both_vort}, but used $\mu_2=0$ for figures \ref{justvel_vel}-\ref{justvel_vort}.  As expected due to the plots in figure \ref{errorplot}, when $\mu_1=\mu_2=10$ we observe rapid convergence in the plots for VV-DA velocity and vorticity to the DNS velocity and vorticity, and by $t=1$ the contour plots are visually indistinguishable.  This is not the case, however, when $\mu_2=0$.  In this case, while the velocity plots do agree with DNS velocities by $t=1$ (in the eyeball norm), the vorticity error remains observable at $t=10$ and even at $t=20$ there are some small difference.  The contour plots of the errors at early times for vorticity show the largest errors occur near vortex centers, indicating that the VV-DA method is not accurately predicting the strength of the vortices.

\begin{figure}[H]
	\centering
	\includegraphics[width = .3\textwidth, height=.15\textwidth]{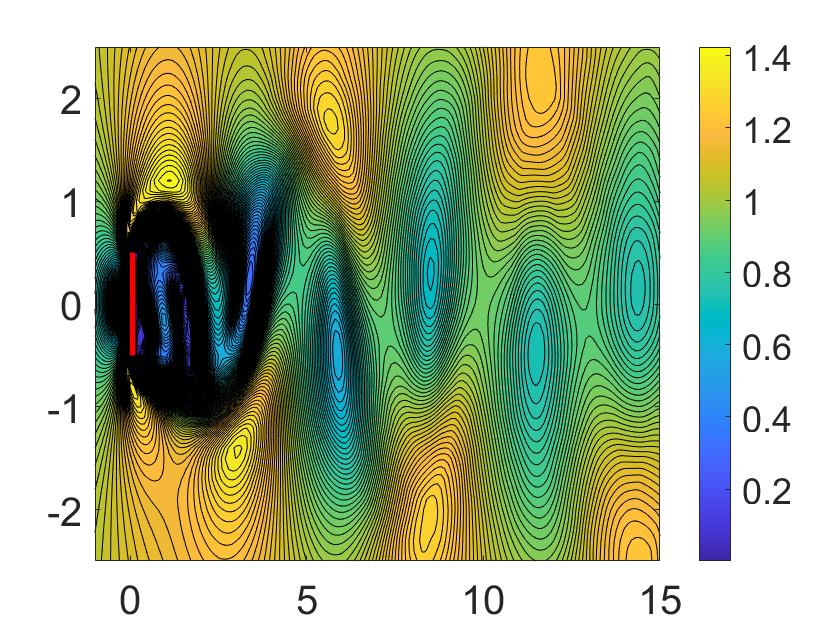}
	\includegraphics[width = .3\textwidth, height=.15\textwidth]{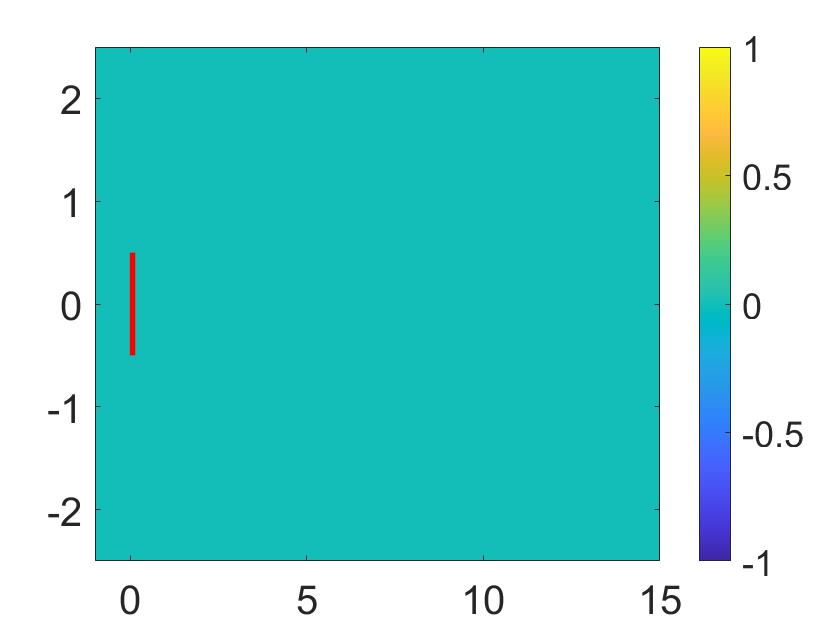}
			\quad\quad\quad\quad\quad\quad\quad\quad\quad\quad\quad\quad\quad\quad\quad
\\
	\includegraphics[width = .3\textwidth, height=.15\textwidth]{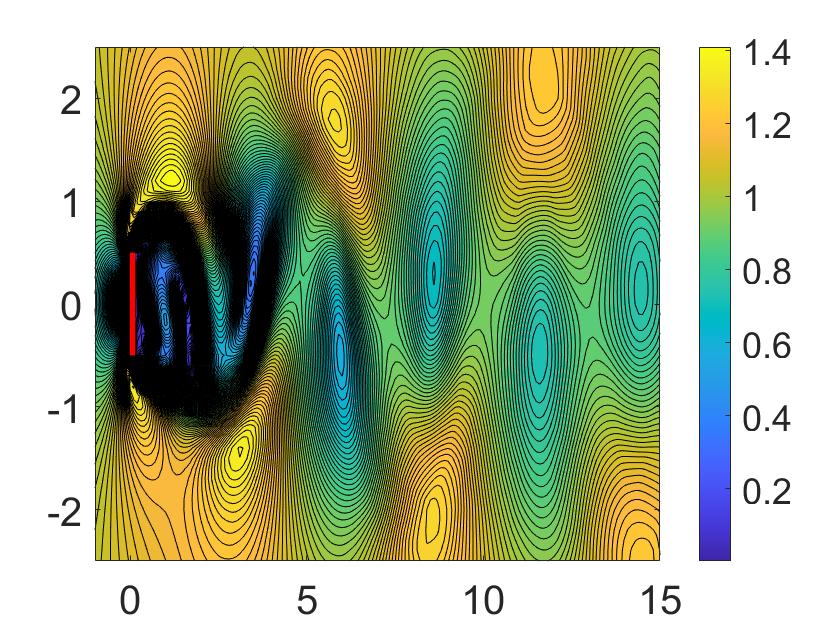}
\includegraphics[width = .3\textwidth, height=.15\textwidth]{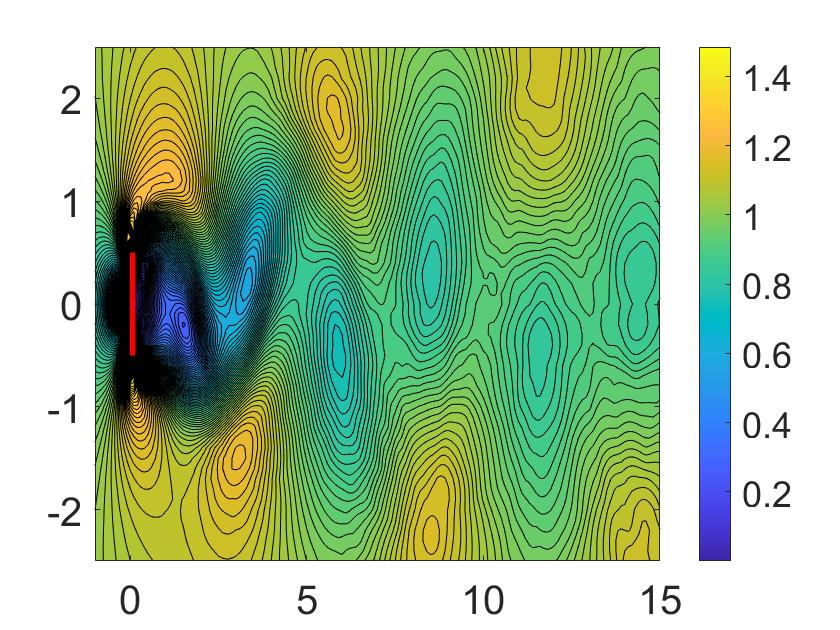}
\includegraphics[width = .3\textwidth, height=.15\textwidth]{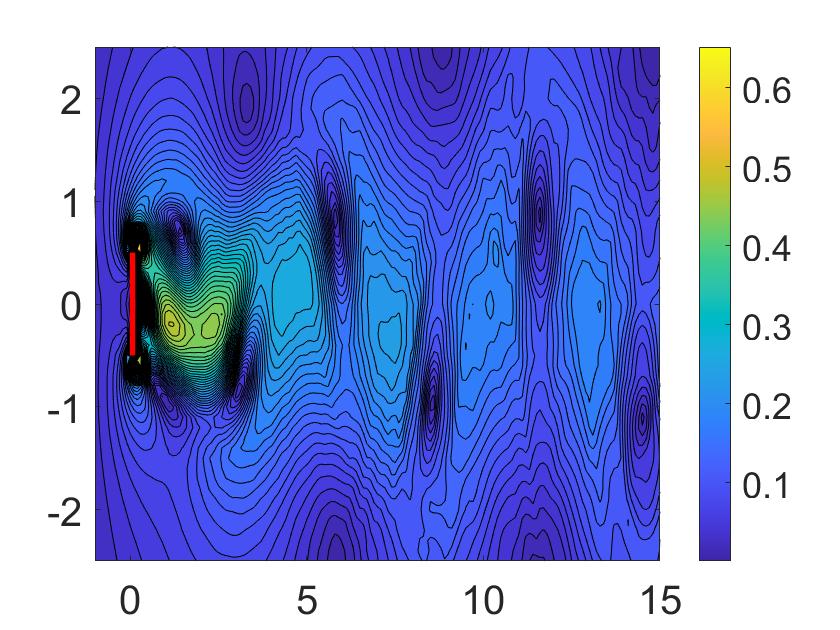}
\\
	\includegraphics[width = .3\textwidth, height=.15\textwidth]{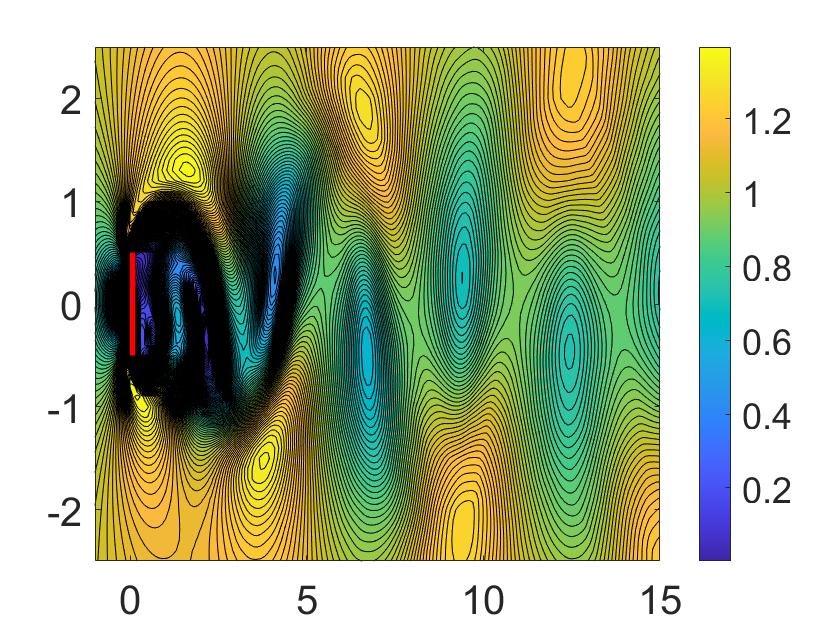}
\includegraphics[width = .3\textwidth, height=.15\textwidth]{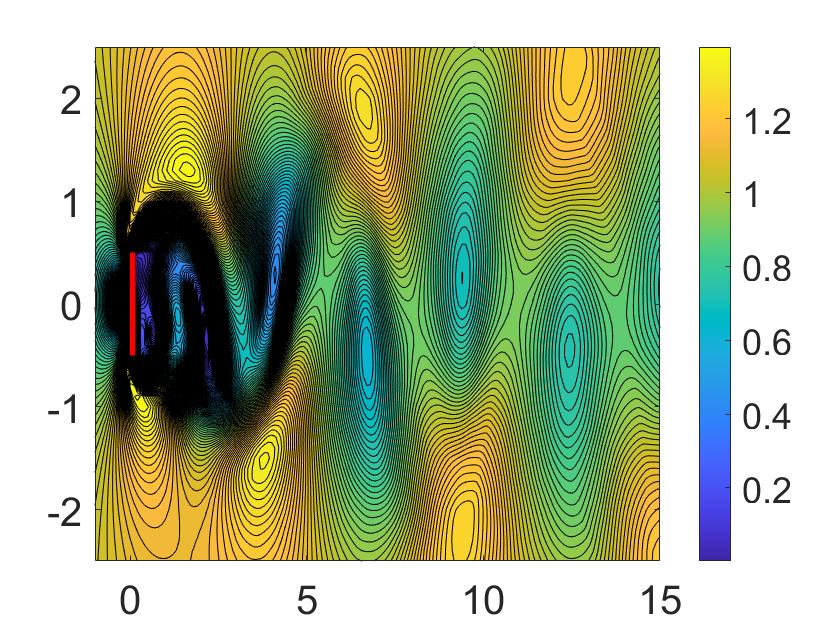}
\includegraphics[width = .3\textwidth, height=.15\textwidth]{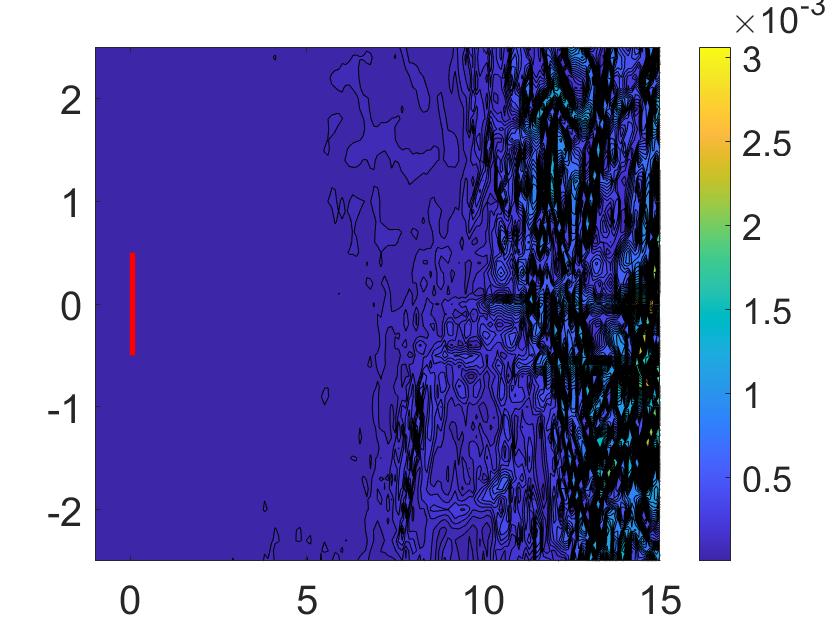}
\\
	\includegraphics[width = .3\textwidth, height=.15\textwidth]{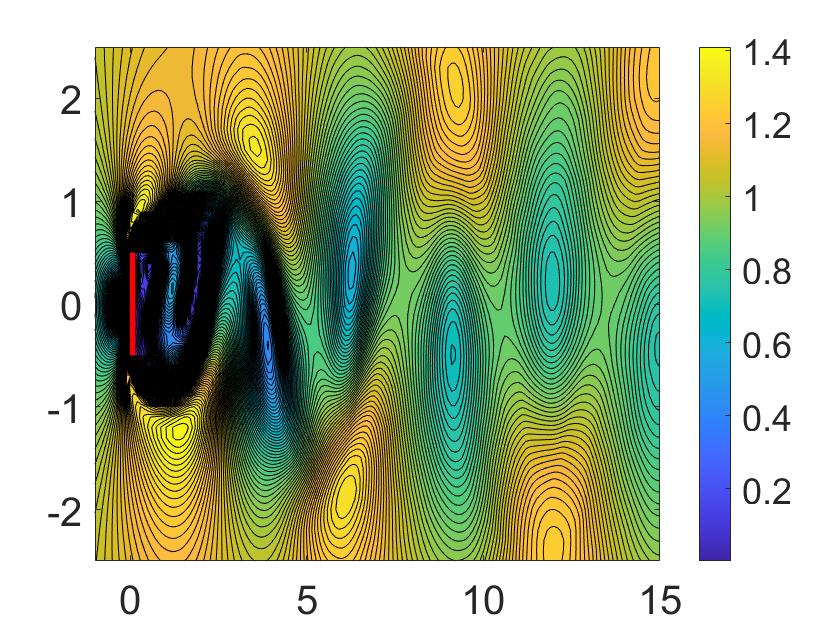}
\includegraphics[width = .3\textwidth, height=.15\textwidth]{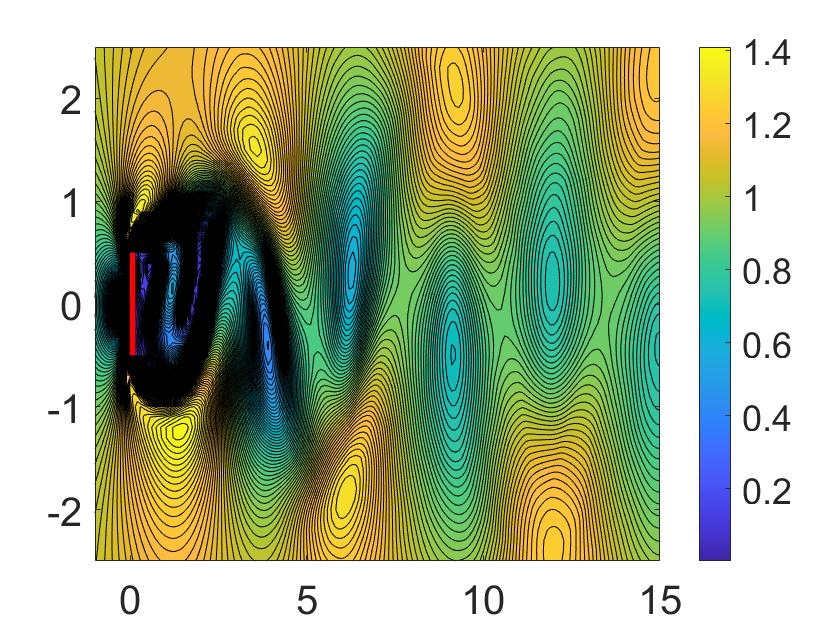}
\includegraphics[width = .3\textwidth, height=.15\textwidth]{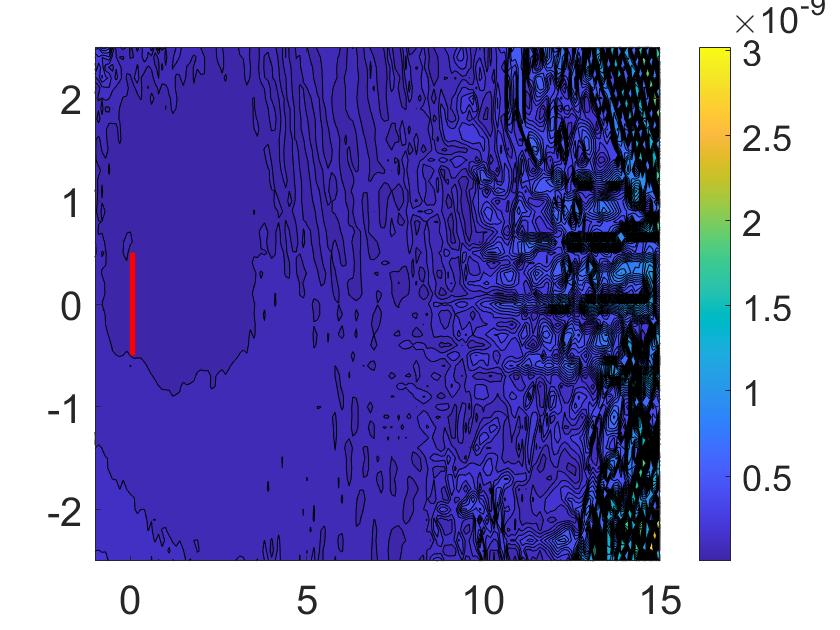}
\\
	\includegraphics[width = .3\textwidth, height=.15\textwidth]{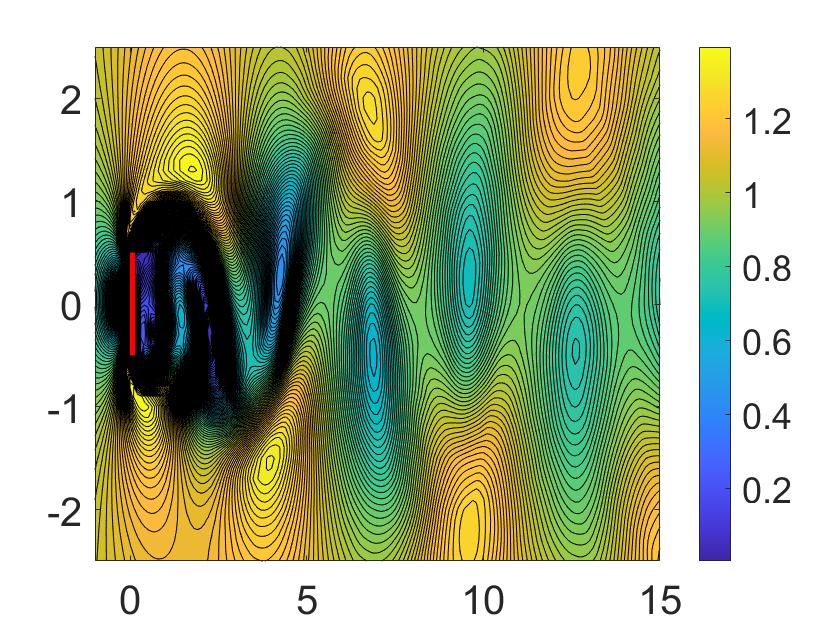}
\includegraphics[width = .3\textwidth, height=.15\textwidth]{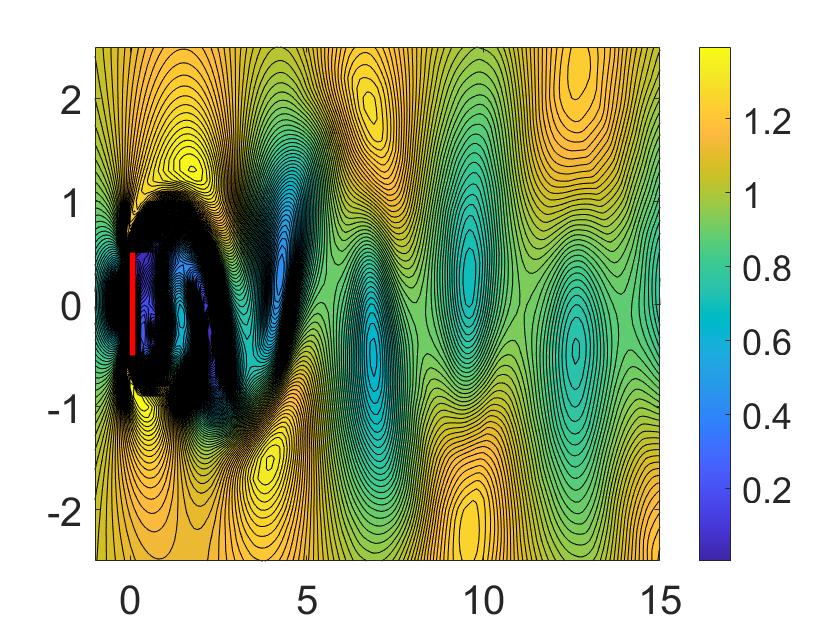}
\includegraphics[width = .3\textwidth, height=.15\textwidth]{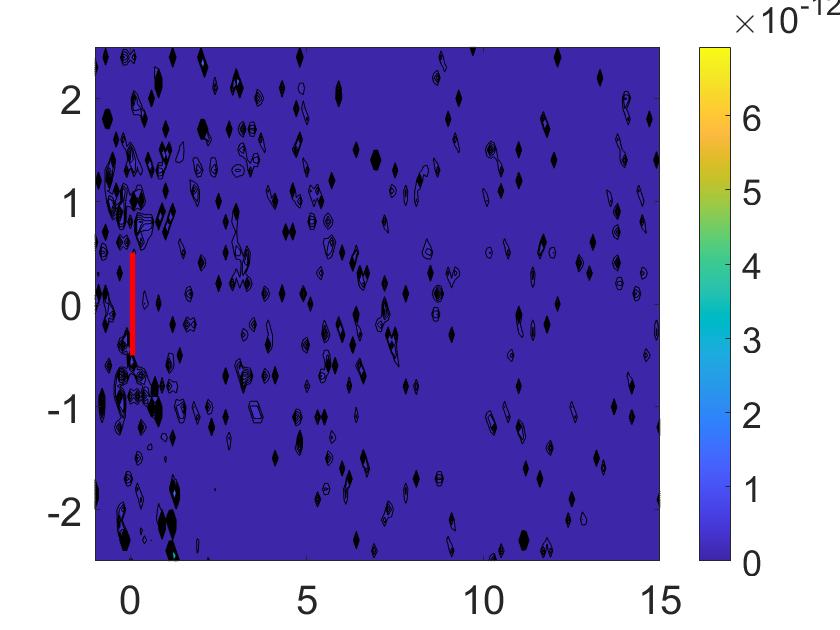}
\\
   	\includegraphics[width = .3\textwidth, height=.15\textwidth]{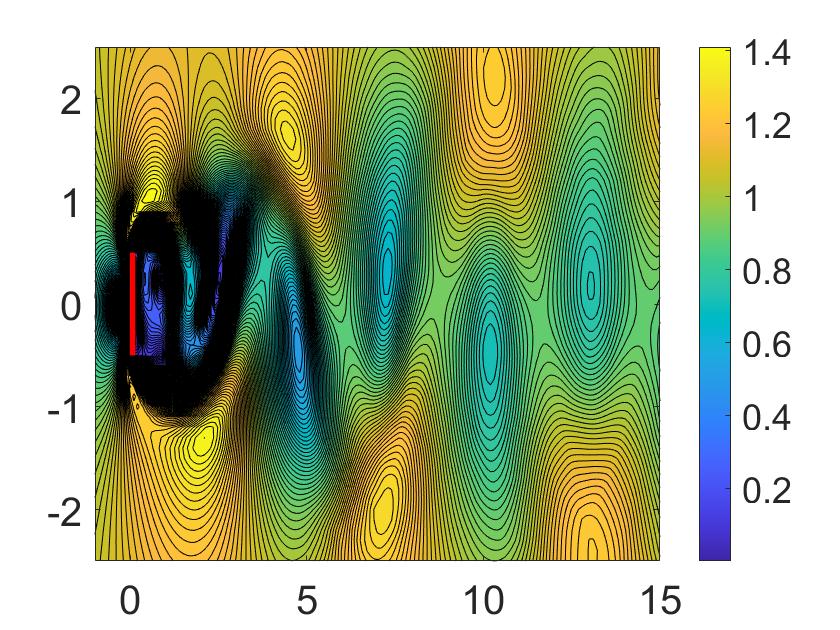}
   \includegraphics[width = .3\textwidth, height=.15\textwidth]{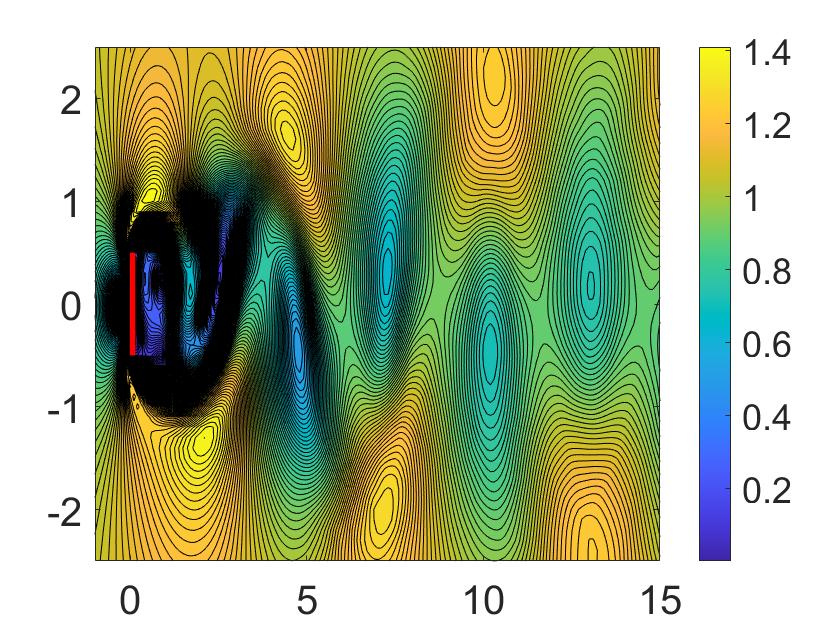}
   \includegraphics[width = .3\textwidth, height=.15\textwidth]{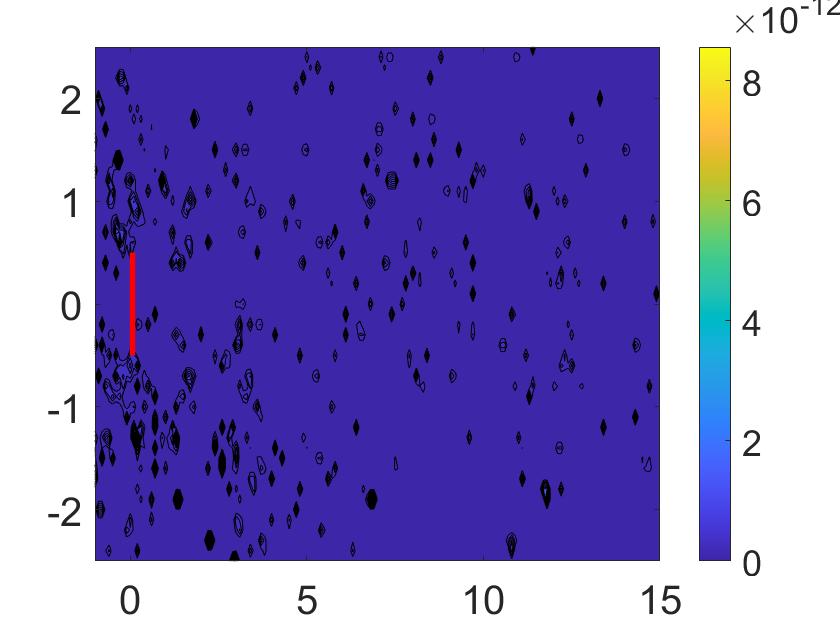}
	\\
	\caption{Contour plots of velocity for DNS (left), VV-DA with $\mu_1=\mu_2=10$ (center), and their difference (right), for times $t=0,\ 0.1,\ 1,\ 10,\ 20,\ 80$ (top to bottom).} \label{both_vel}
\end{figure}

\begin{figure}[H]
	\centering
	\includegraphics[width = .3\textwidth, height=.15\textwidth]{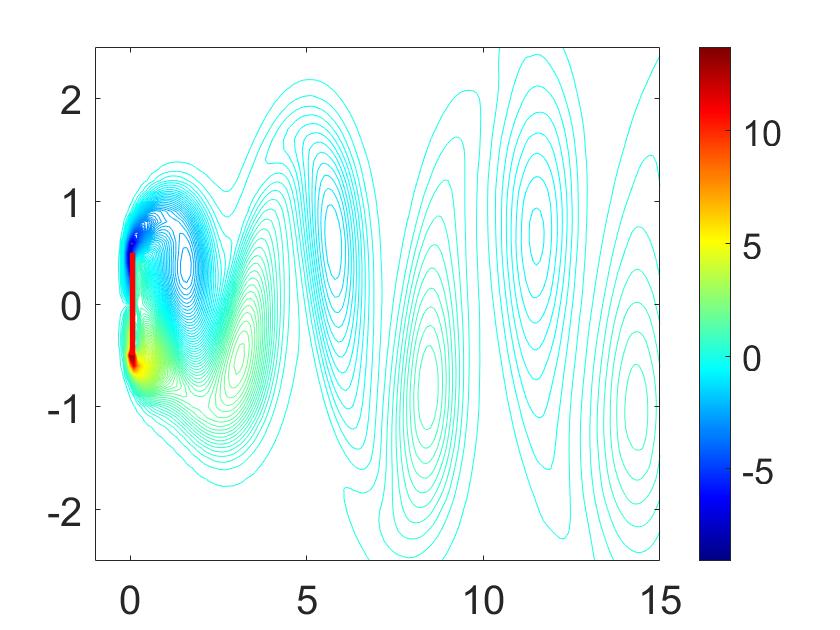}
	\includegraphics[width = .3\textwidth, height=.15\textwidth]{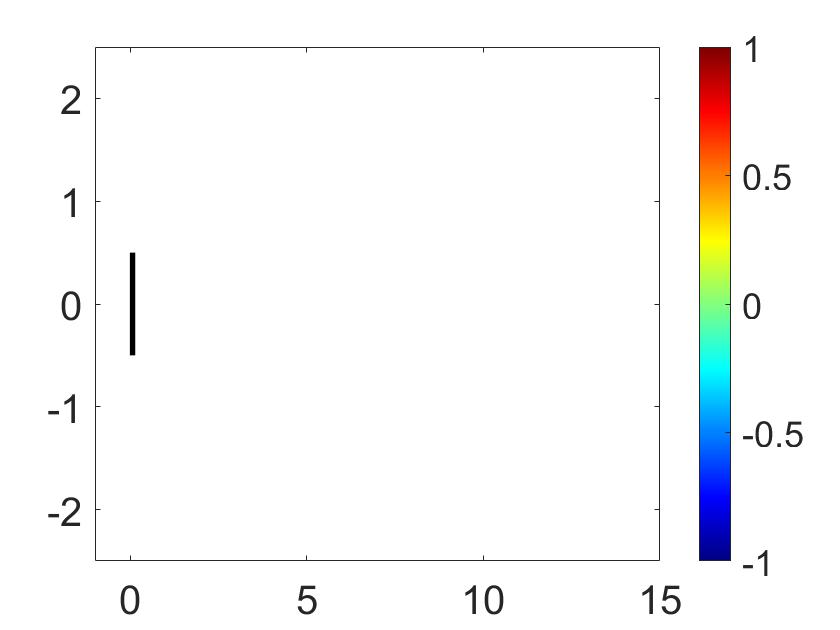}
			\quad\quad\quad\quad\quad\quad\quad\quad\quad\quad\quad\quad\quad\quad\quad
	\\
	\includegraphics[width = .3\textwidth, height=.15\textwidth]{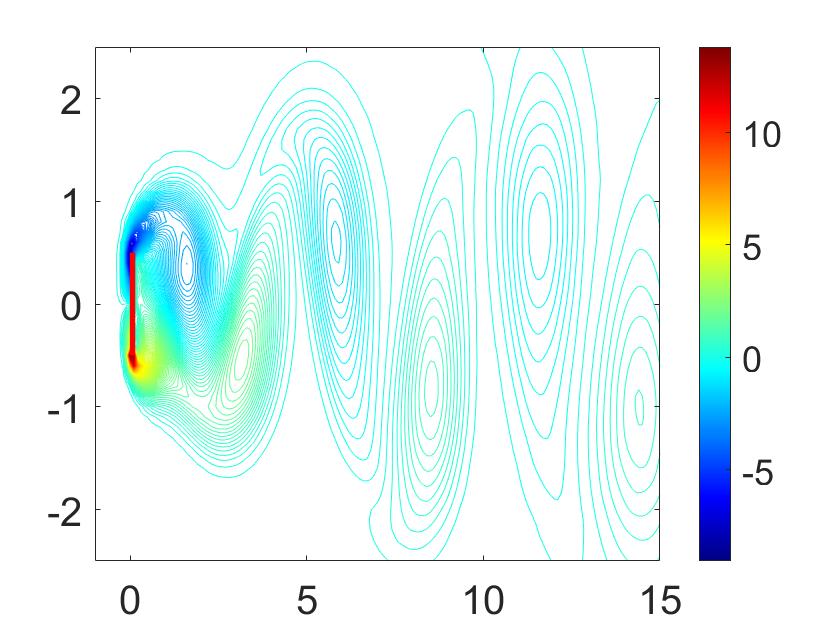}
	\includegraphics[width = .3\textwidth, height=.15\textwidth]{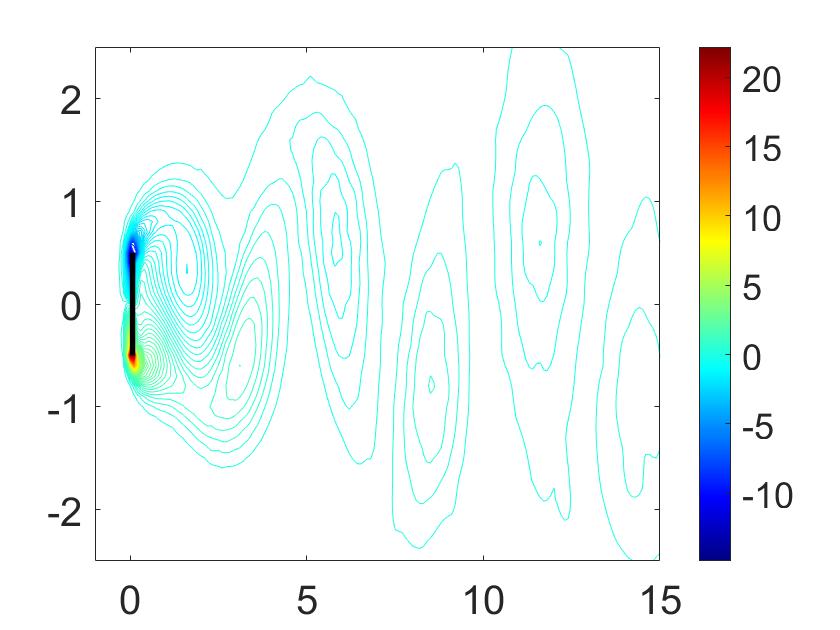}
	\includegraphics[width = .3\textwidth, height=.15\textwidth]{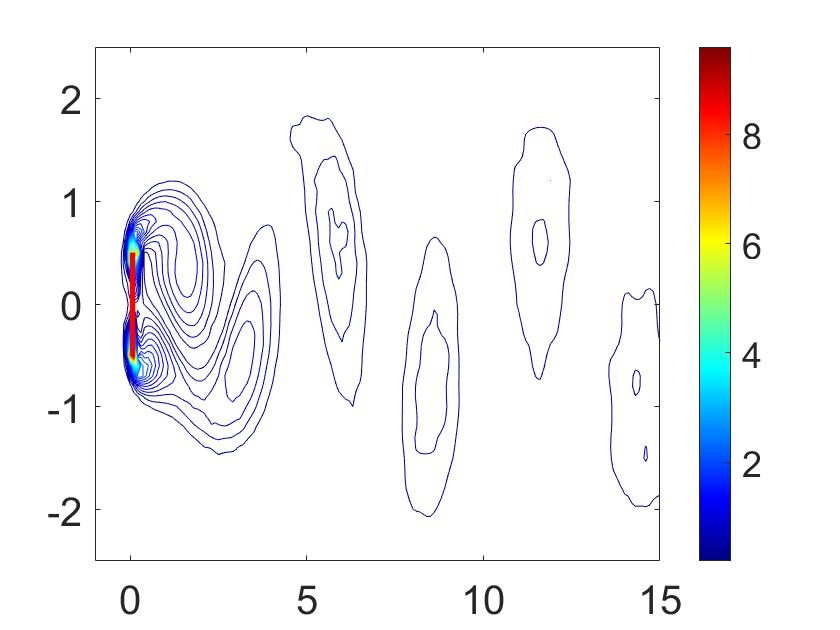}
	\\
	\includegraphics[width = .3\textwidth, height=.15\textwidth]{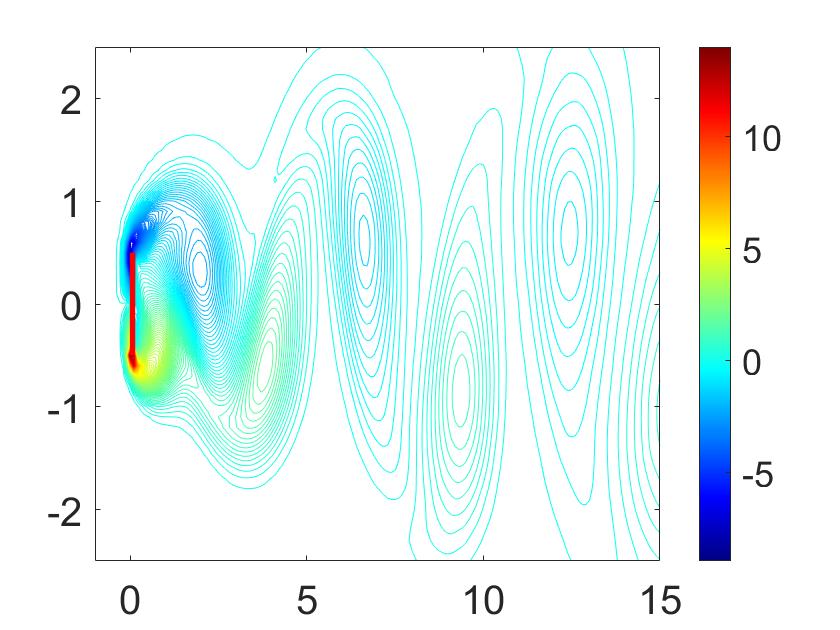}
	\includegraphics[width = .3\textwidth, height=.15\textwidth]{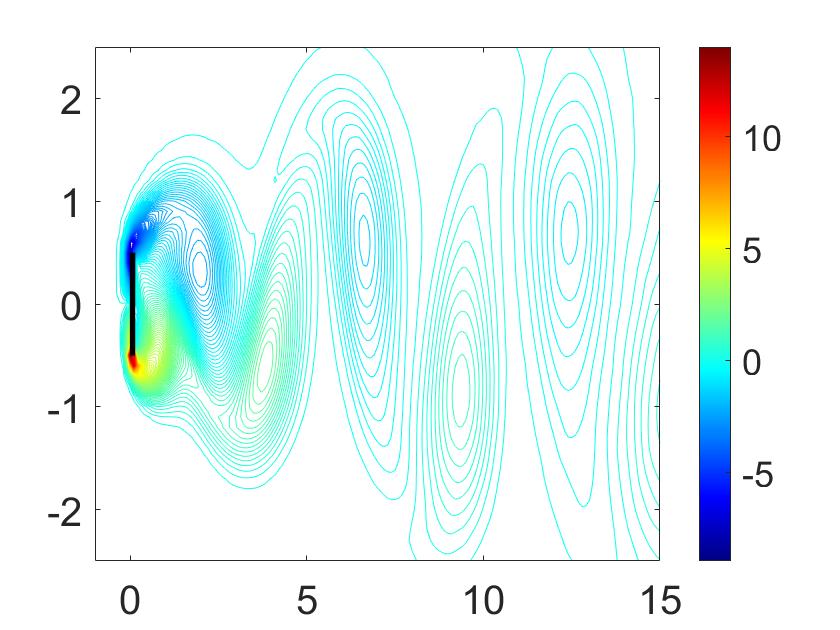}
	\includegraphics[width = .3\textwidth, height=.15\textwidth]{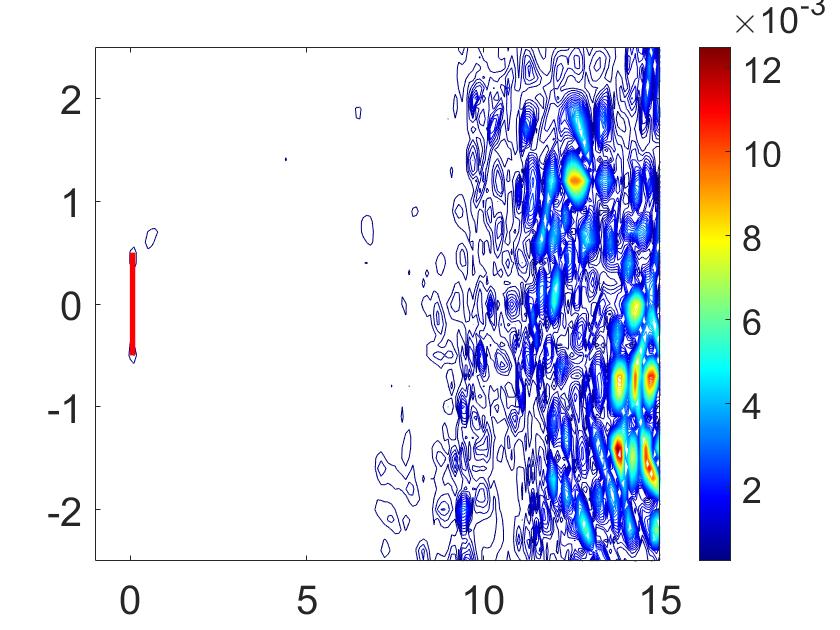}
	\\
	\includegraphics[width = .3\textwidth, height=.15\textwidth]{vort_DNS_just_vel_4050.jpg}
	\includegraphics[width = .3\textwidth, height=.15\textwidth]{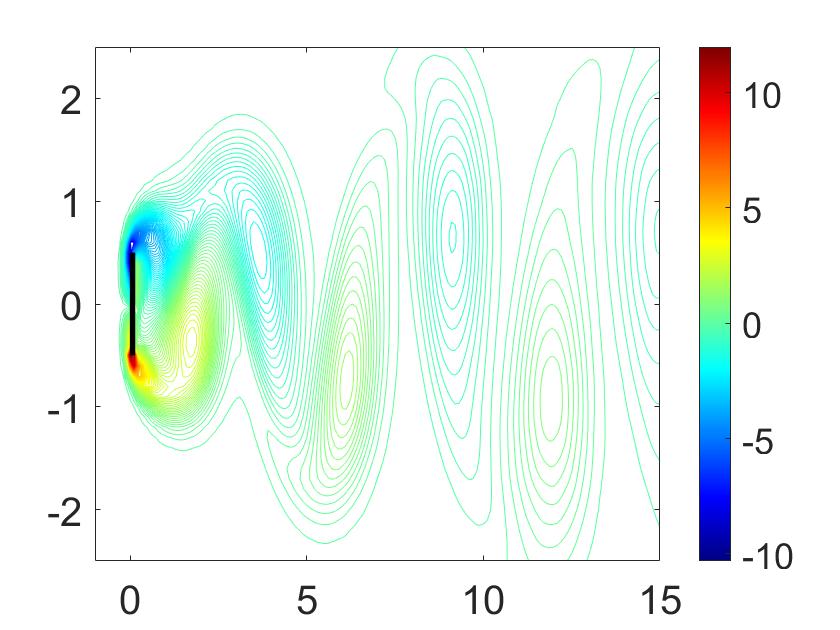}
	\includegraphics[width = .3\textwidth, height=.15\textwidth]{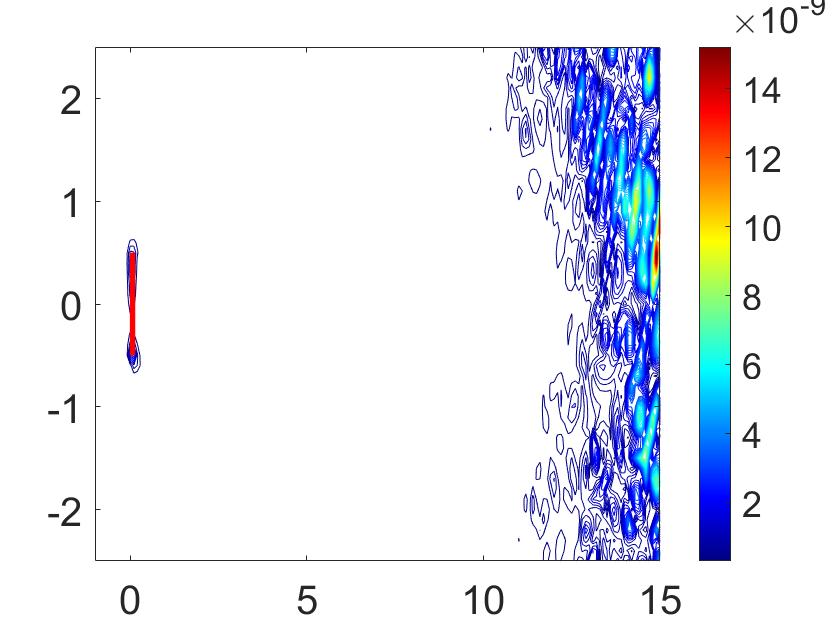}
	\\
	\includegraphics[width = .3\textwidth, height=.15\textwidth]{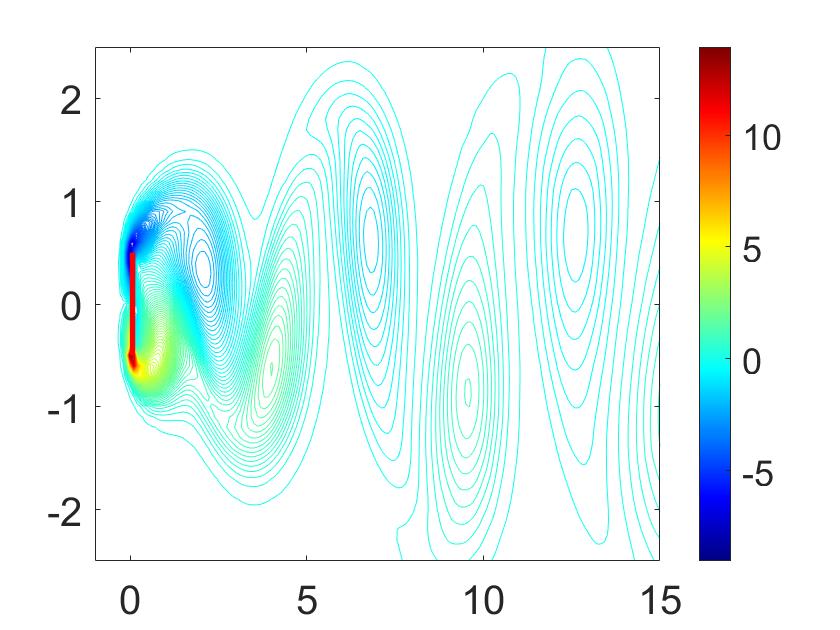}
	\includegraphics[width = .3\textwidth, height=.15\textwidth]{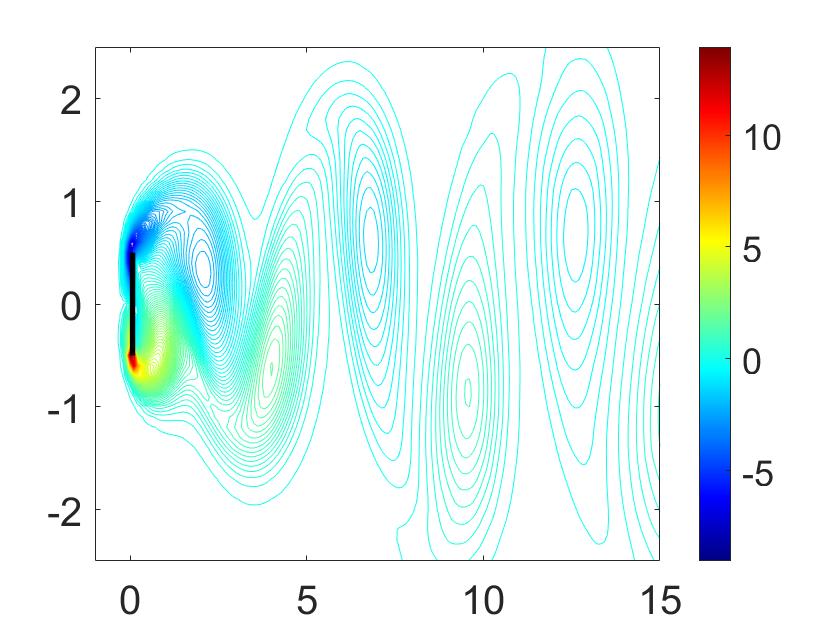}
	\includegraphics[width = .3\textwidth, height=.15\textwidth]{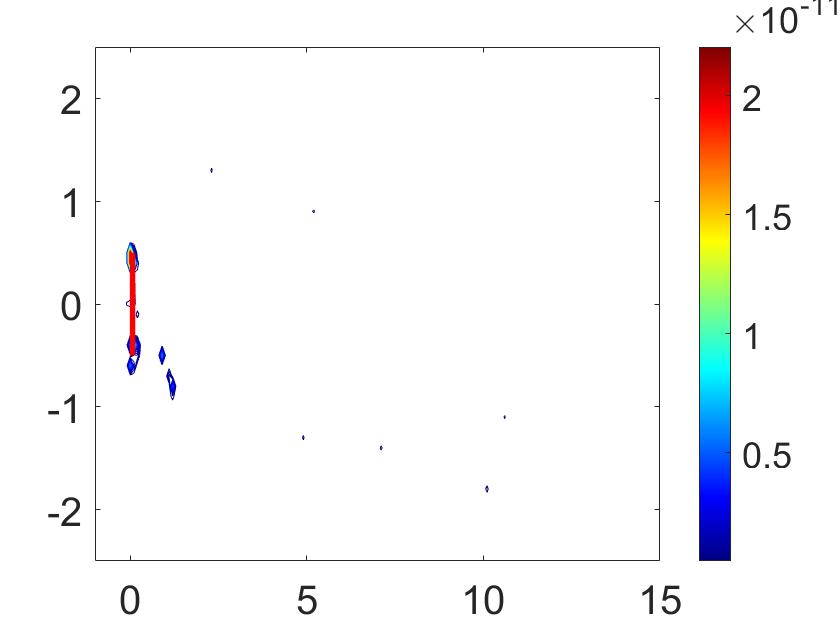}
	\\
	\includegraphics[width = .3\textwidth, height=.15\textwidth]{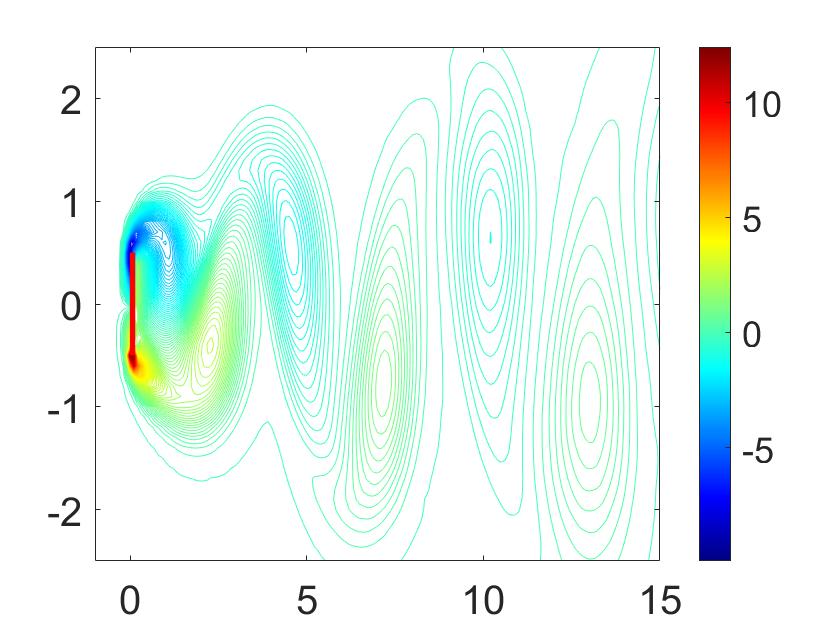}
	\includegraphics[width = .3\textwidth, height=.15\textwidth]{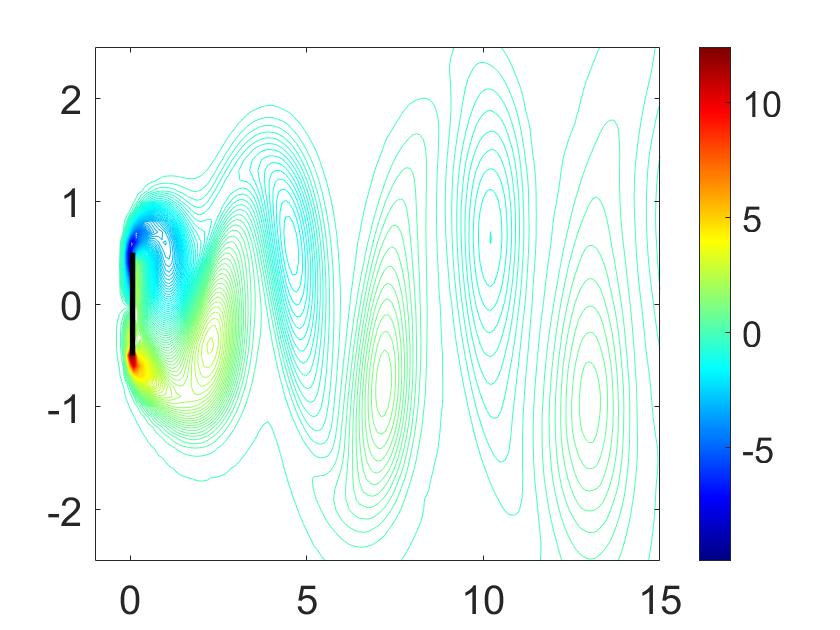}
	\includegraphics[width = .3\textwidth, height=.15\textwidth]{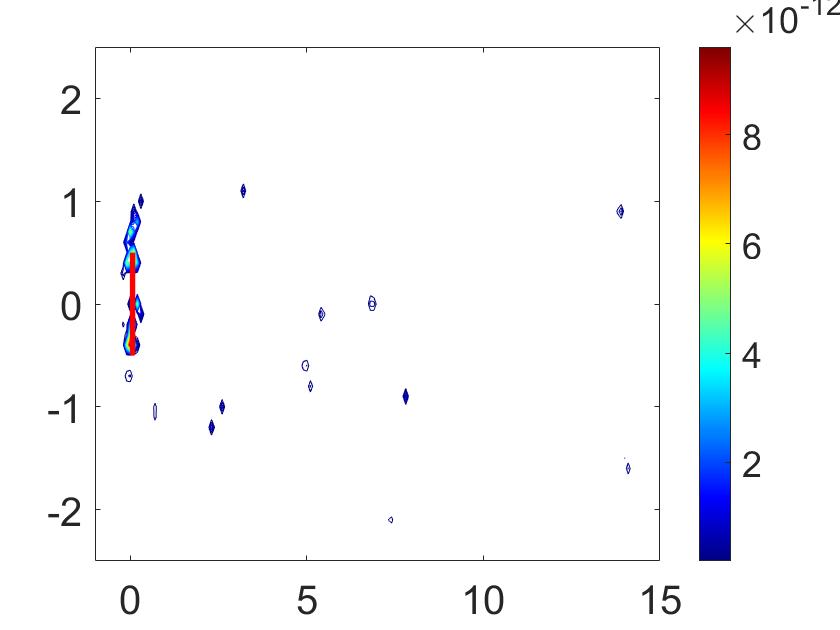}
	\\
	\caption{Contour plots of vorticity for DNS (left), VV-DA with $\mu_1=\mu_2=10$ (center), and their difference (right), for times $t=0,\ 0.1,\ 1,\ 10,\ 20,\ 80$ (top to bottom).} \label{both_vort}
\end{figure}

\begin{figure}[H]
	\centering
	\includegraphics[width = .3\textwidth, height=.15\textwidth]{vel_DNS_4000.jpg}
	\includegraphics[width = .3\textwidth, height=.15\textwidth]{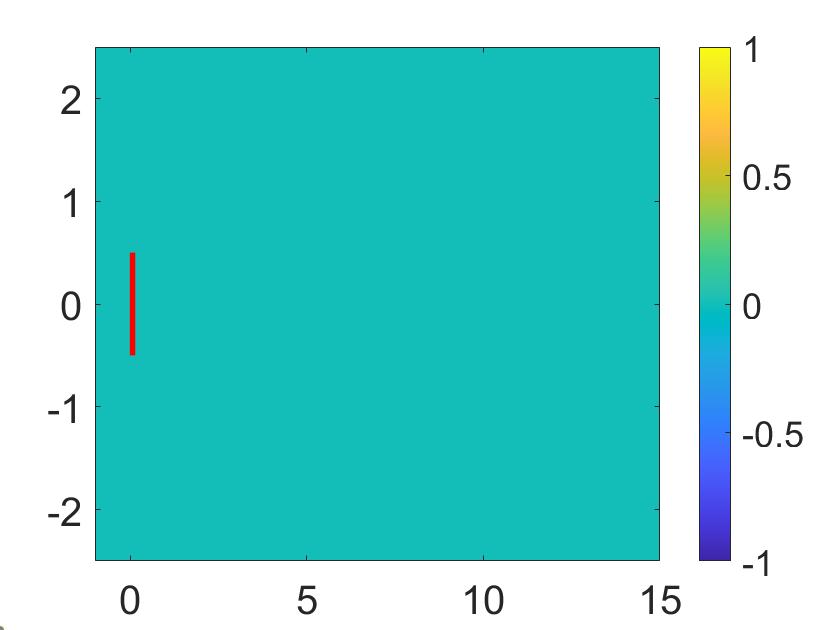}
			\quad\quad\quad\quad\quad\quad\quad\quad\quad\quad\quad\quad\quad\quad\quad
	\\
	\includegraphics[width = .3\textwidth, height=.15\textwidth]{vel_DNS_4005.jpg}
\includegraphics[width = .3\textwidth, height=.15\textwidth]{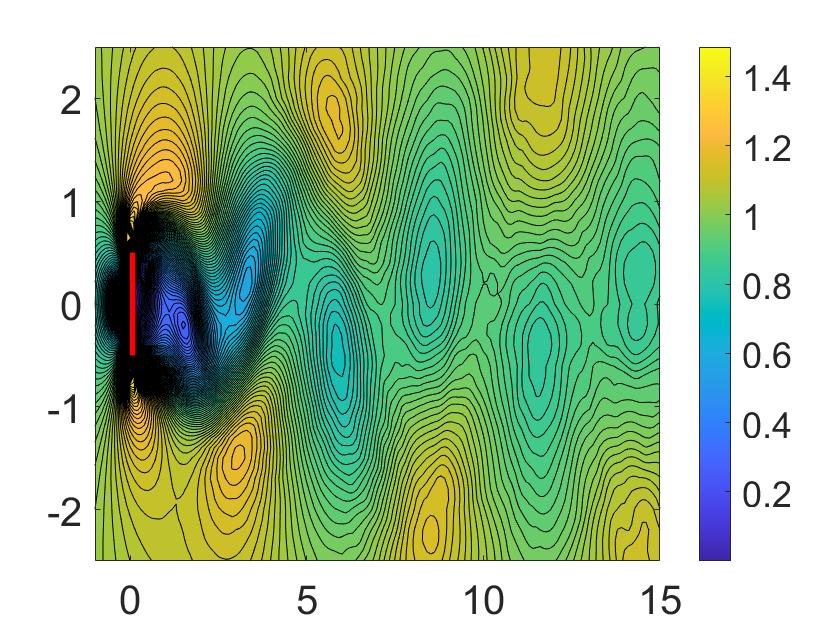}
\includegraphics[width = .3\textwidth, height=.15\textwidth]{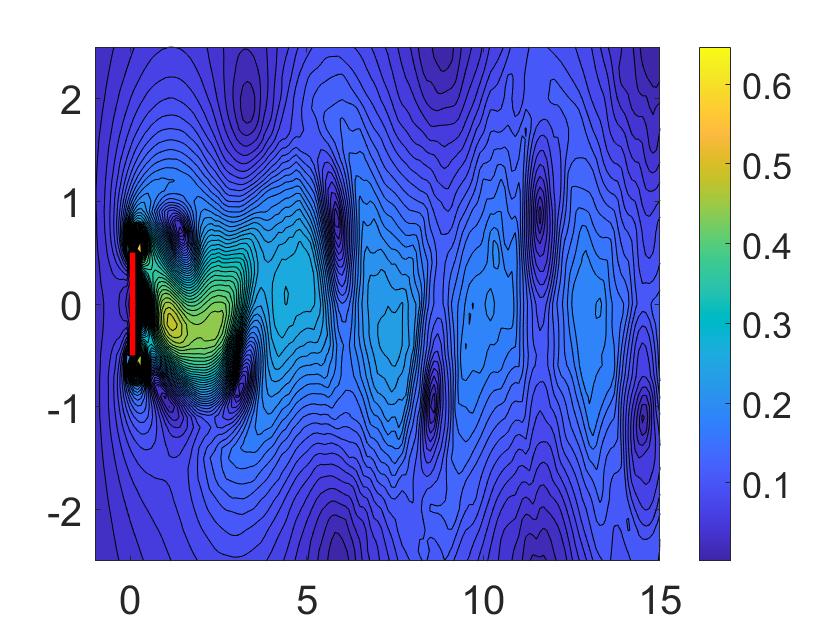}
	\\
	\includegraphics[width = .3\textwidth, height=.15\textwidth]{vel_DNS_4050.jpg}
\includegraphics[width = .3\textwidth, height=.15\textwidth]{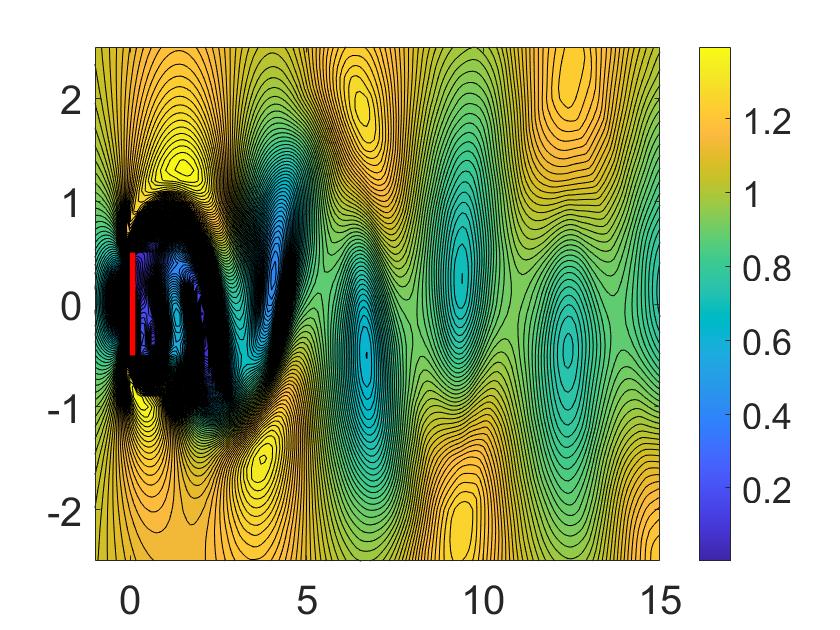}
\includegraphics[width = .3\textwidth, height=.15\textwidth]{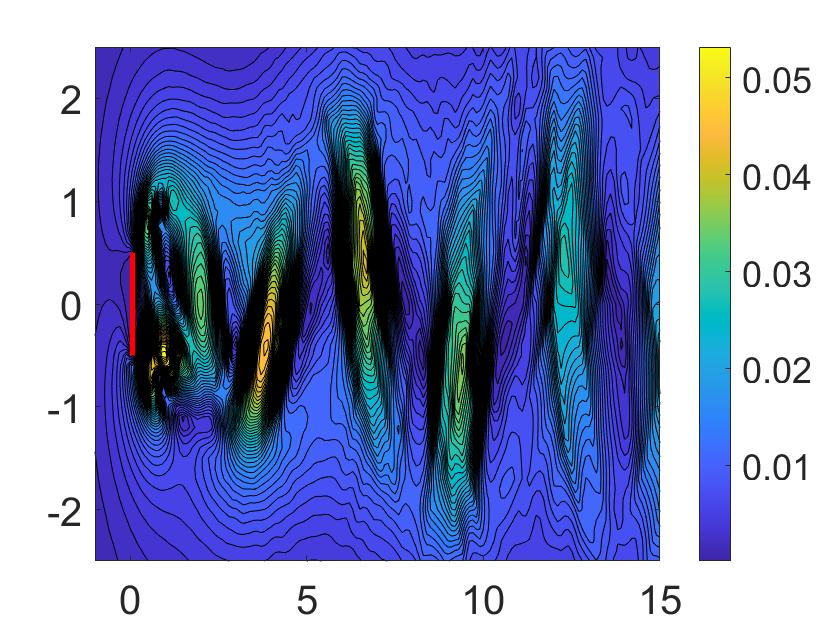}
	\\
	\includegraphics[width = .3\textwidth, height=.15\textwidth]{vel_DNS_4500.jpg}
\includegraphics[width = .3\textwidth, height=.15\textwidth]{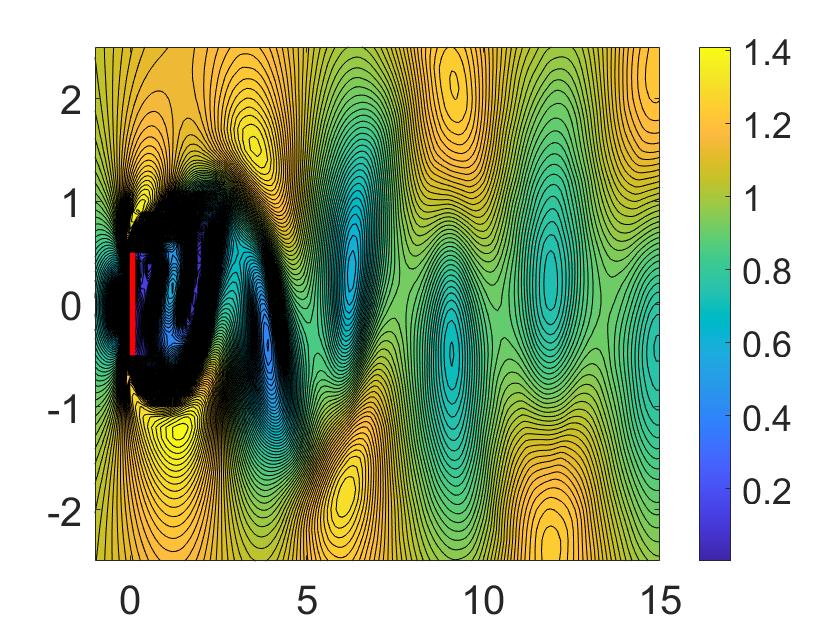}
\includegraphics[width = .3\textwidth, height=.15\textwidth]{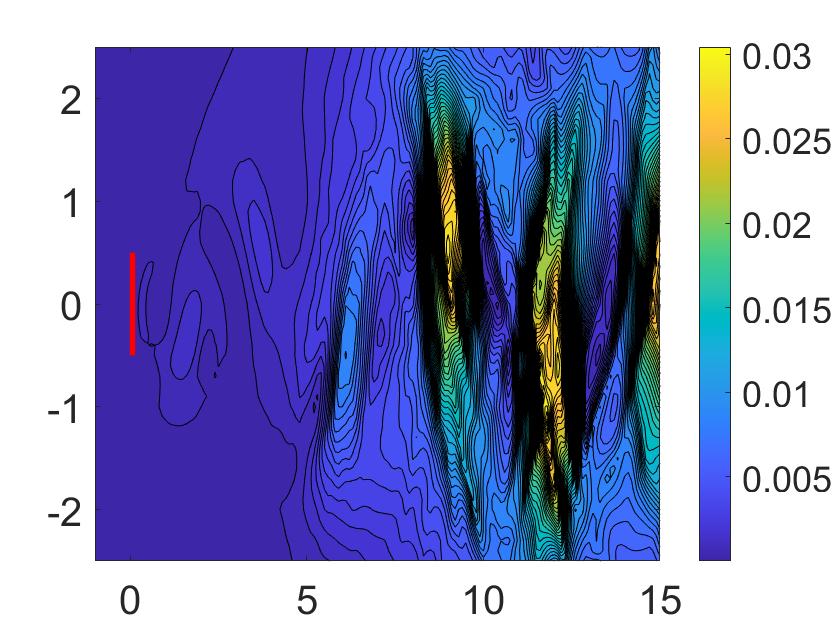}
	\\
	\includegraphics[width = .3\textwidth, height=.15\textwidth]{vel_DNS_5000.jpg}
\includegraphics[width = .3\textwidth, height=.15\textwidth]{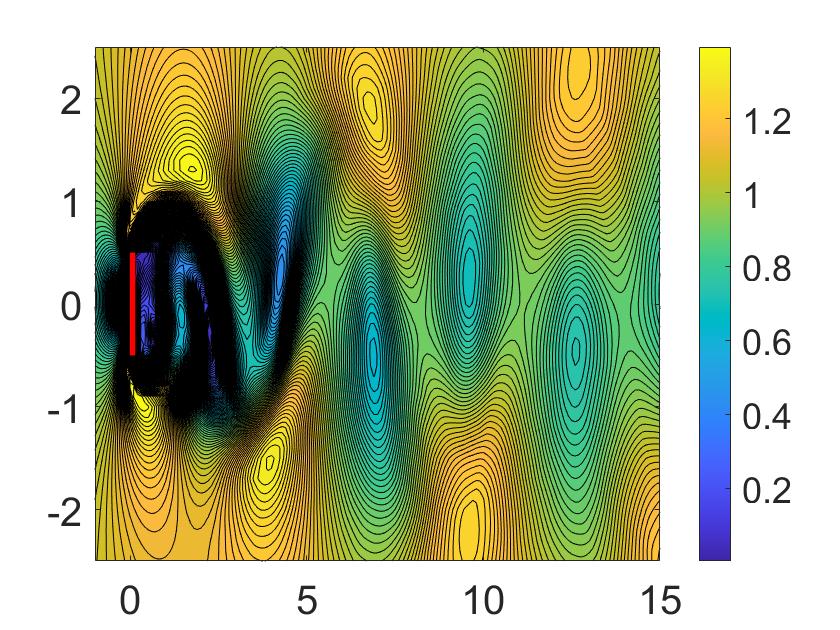}
\includegraphics[width = .3\textwidth, height=.15\textwidth]{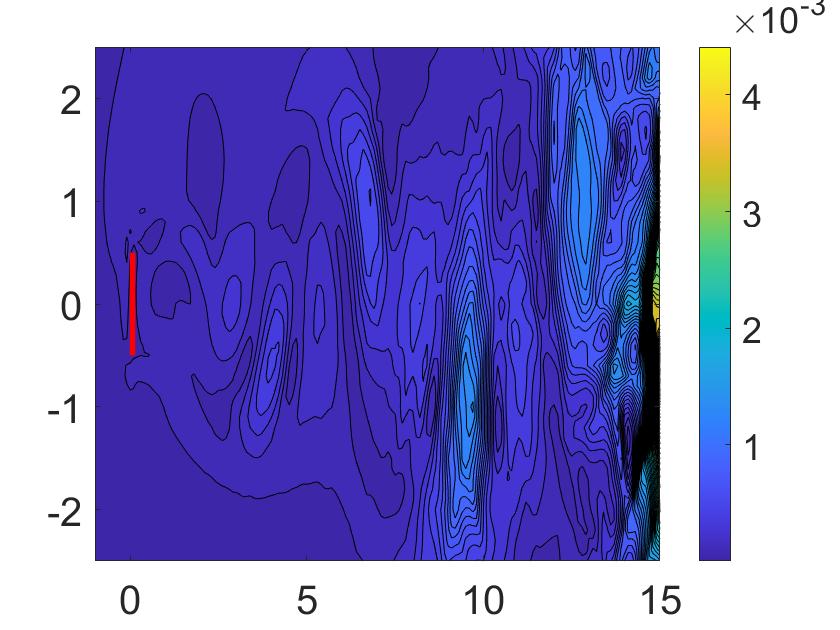}
	\\
	\includegraphics[width = .3\textwidth, height=.15\textwidth]{vel_DNS_8000.jpg}
\includegraphics[width = .3\textwidth, height=.15\textwidth]{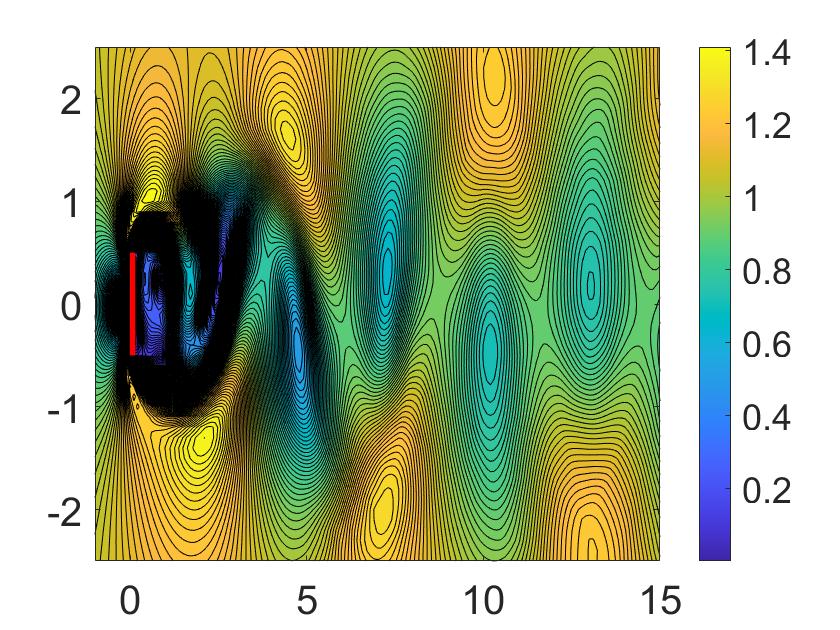}
\includegraphics[width = .3\textwidth, height=.15\textwidth]{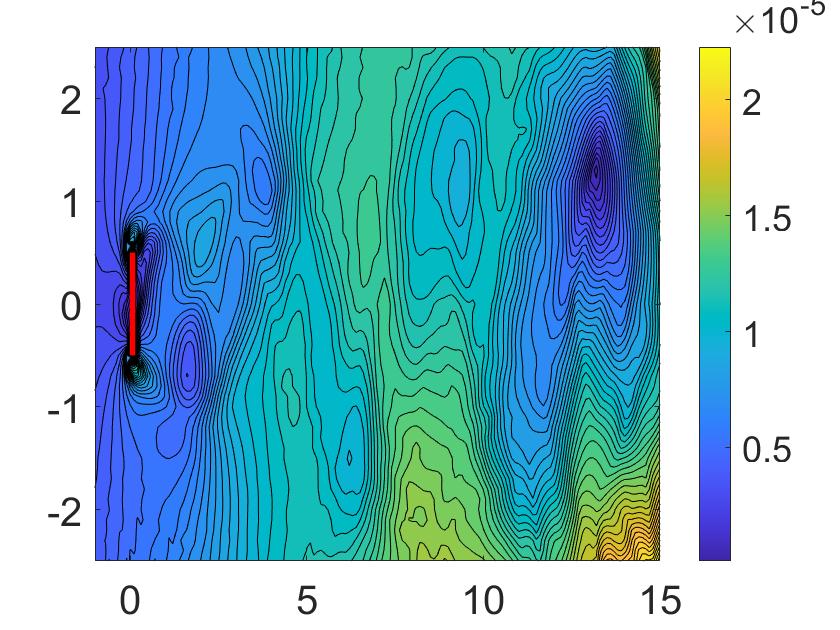}
	\\
	\caption{Contour plots of velocity for DNS (left), VV-DA with $\mu_1=10,\ \mu_2=0$ (center), and their difference (right), for times $t=0,\ 0.1,\ 1,\ 10,\ 20,\ 80$ (top to bottom).} \label{justvel_vel}
\end{figure}

\begin{figure}[H]
	\centering
	\includegraphics[width = .3\textwidth, height=.15\textwidth]{vort_DNS_just_vel_4000.jpg}
\includegraphics[width = .3\textwidth, height=.15\textwidth]{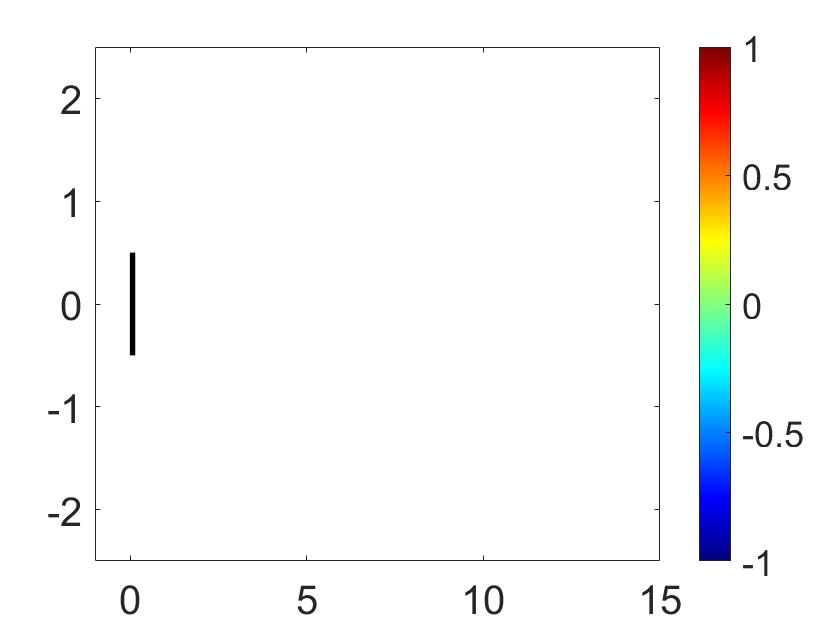}
			\quad\quad\quad\quad\quad\quad\quad\quad\quad\quad\quad\quad\quad\quad\quad
\\
\includegraphics[width = .3\textwidth, height=.15\textwidth]{vort_DNS_just_vel_4005.jpg}
\includegraphics[width = .3\textwidth, height=.15\textwidth]{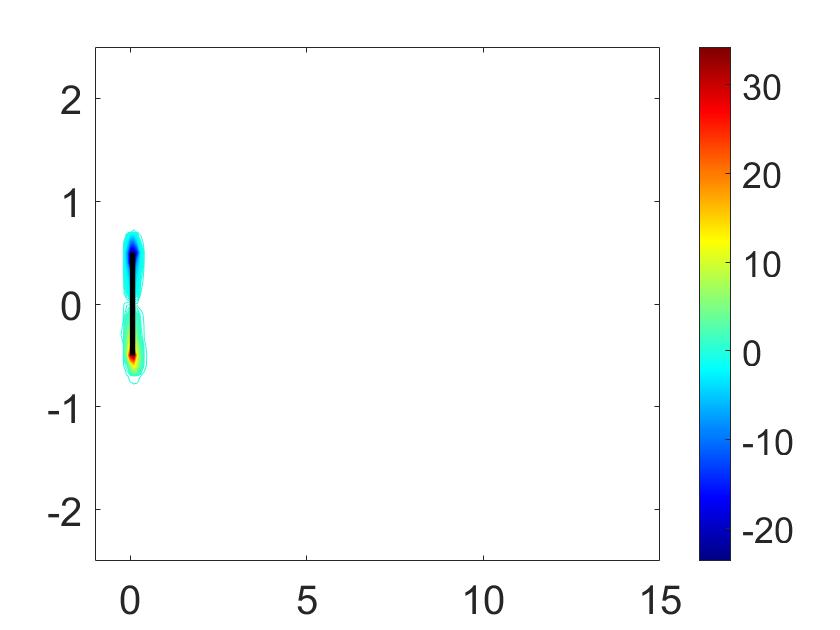}
\includegraphics[width = .3\textwidth, height=.15\textwidth]{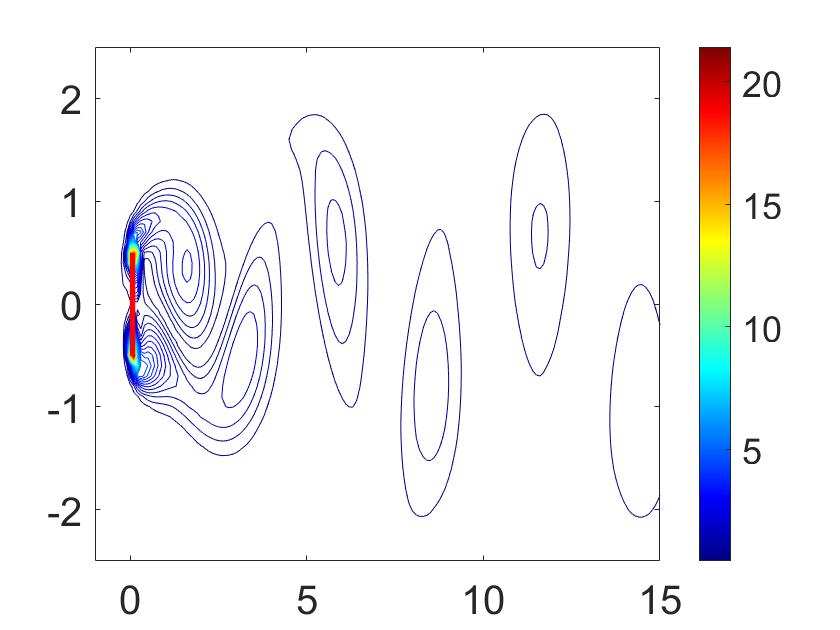}
\\
\includegraphics[width = .3\textwidth, height=.15\textwidth]{vort_DNS_just_vel_4050.jpg}
\includegraphics[width = .3\textwidth, height=.15\textwidth]{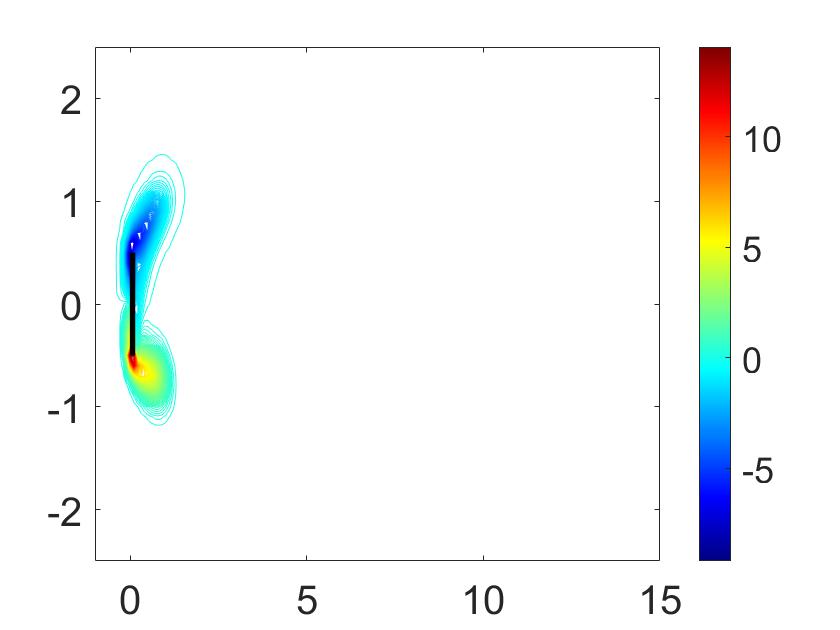}
\includegraphics[width = .3\textwidth, height=.15\textwidth]{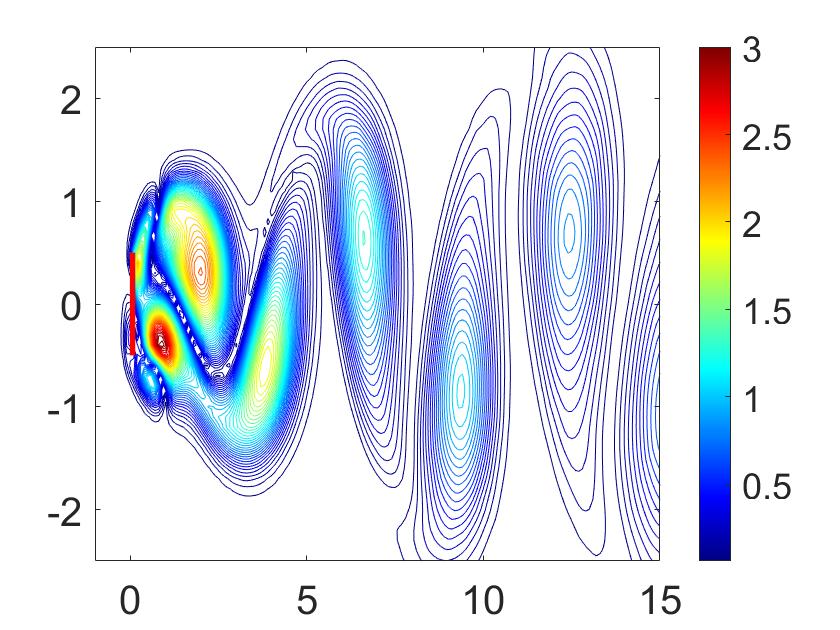}
\\
\includegraphics[width = .3\textwidth, height=.15\textwidth]{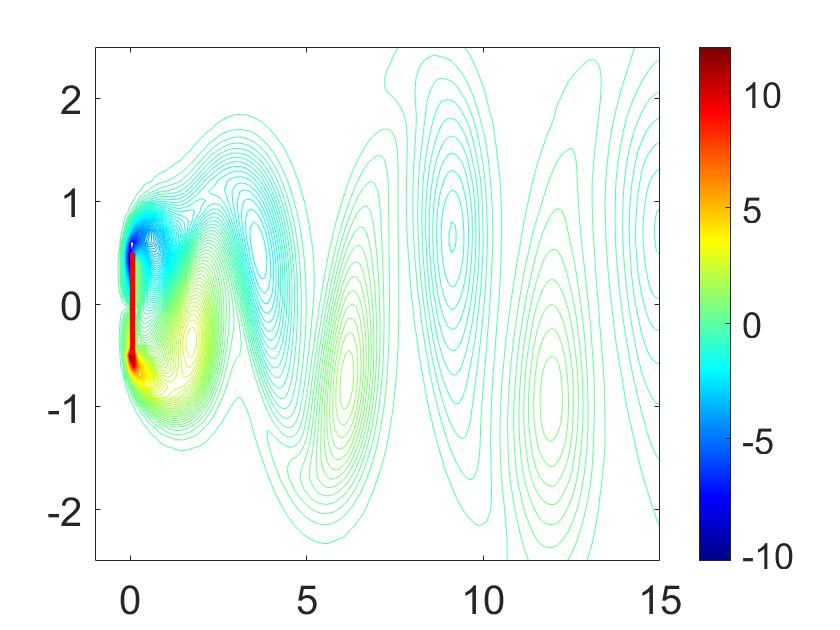}
\includegraphics[width = .3\textwidth, height=.15\textwidth]{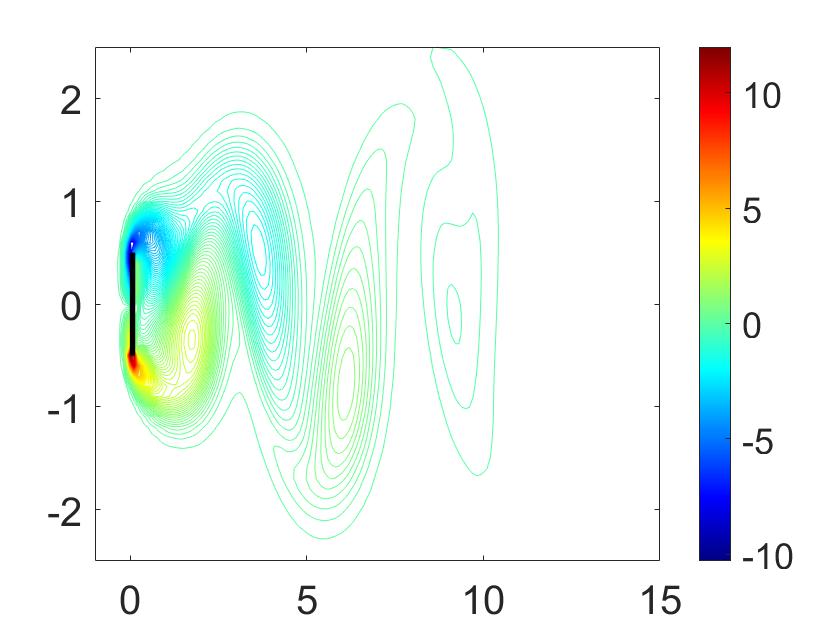}
\includegraphics[width = .3\textwidth, height=.15\textwidth]{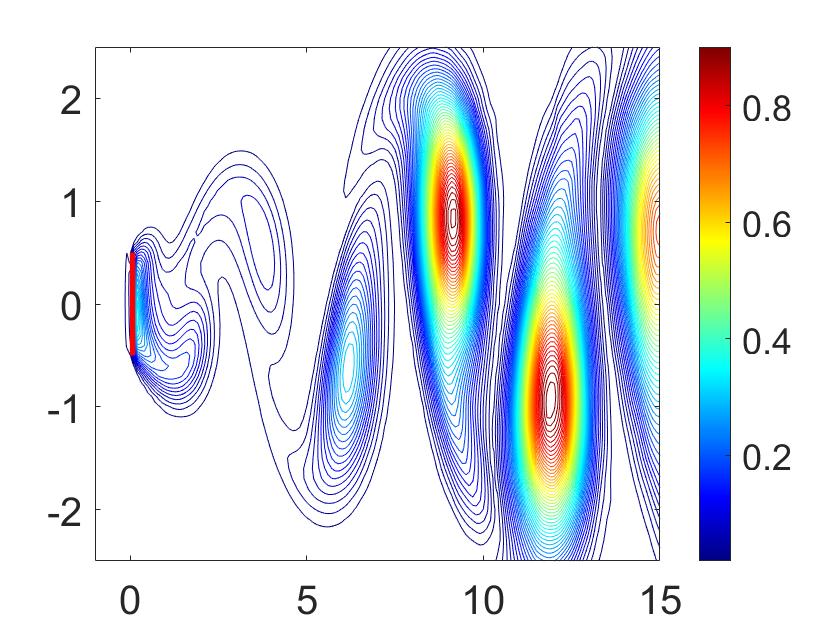}
\\
\includegraphics[width = .3\textwidth, height=.15\textwidth]{vort_DNS_just_vel_5000.jpg}
\includegraphics[width = .3\textwidth, height=.15\textwidth]{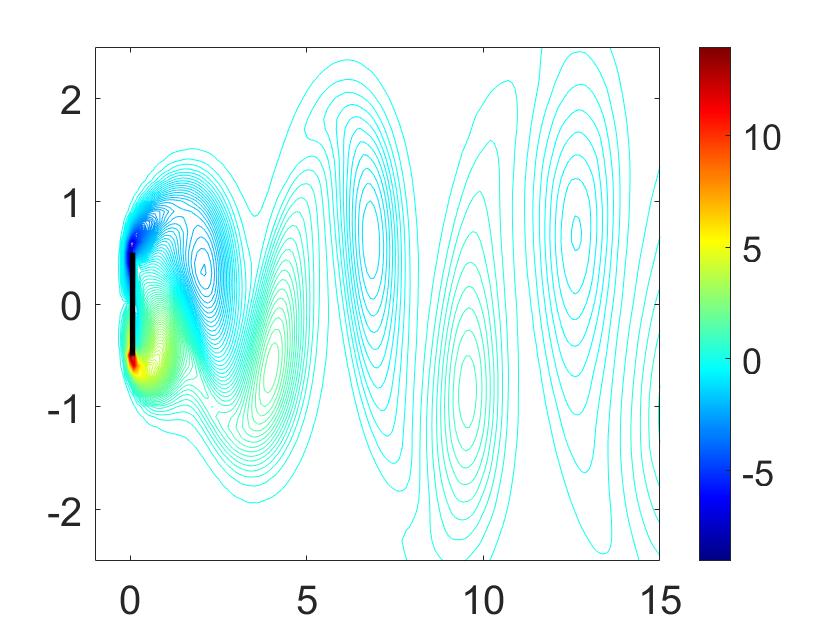}
\includegraphics[width = .3\textwidth, height=.15\textwidth]{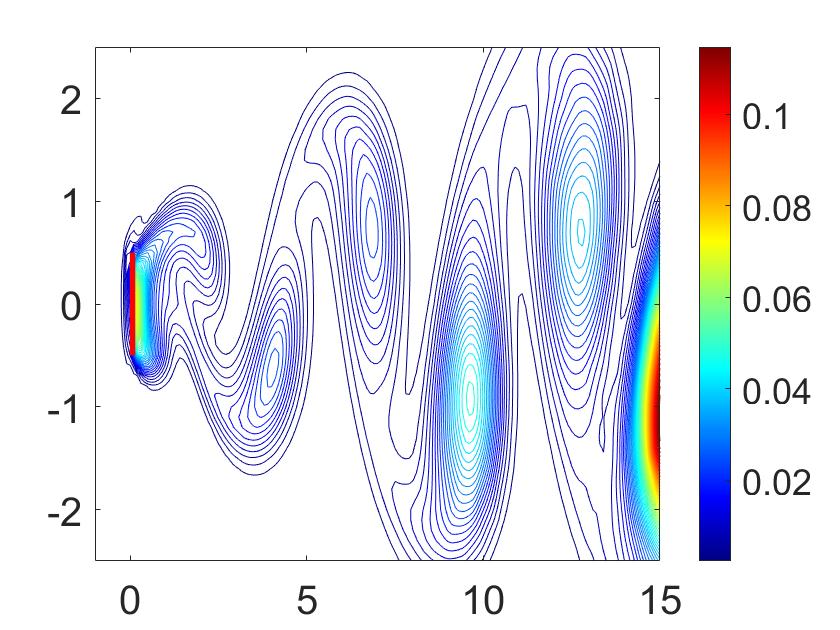}
\\
\includegraphics[width = .3\textwidth, height=.15\textwidth]{vort_DNS_just_vel_8000.jpg}
\includegraphics[width = .3\textwidth, height=.15\textwidth]{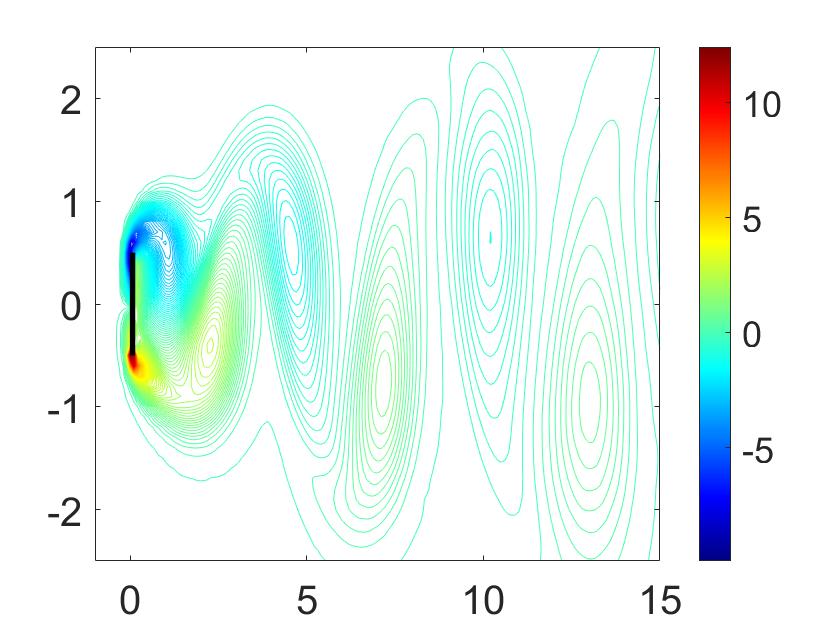}
\includegraphics[width = .3\textwidth, height=.15\textwidth]{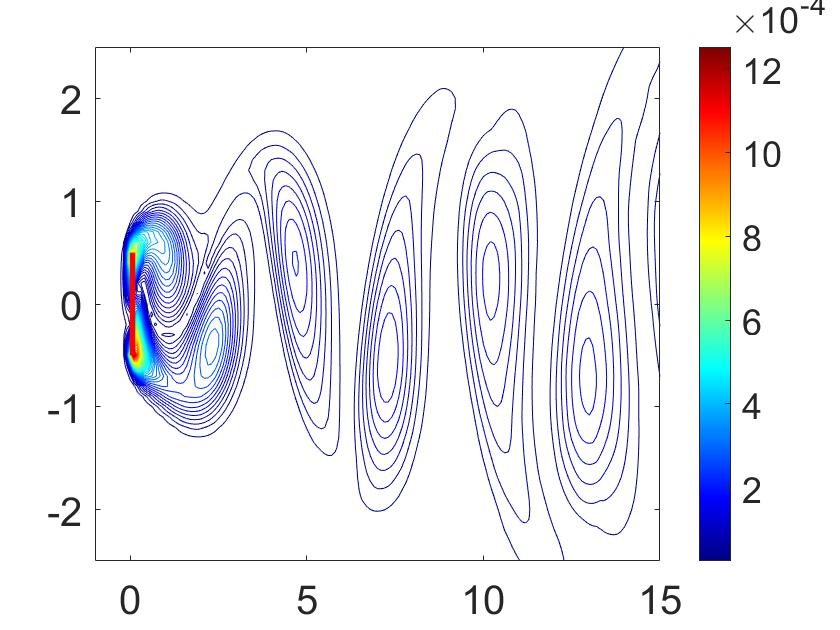}
\\
	\caption{Contour plots of vorticity for DNS (left), VV-DA with $\mu_1=10,\ \mu_2=0$ (center), and their difference (right), for times $t=0,\ 0.1,\ 1,\ 10,\ 20,\ 80$ (top to bottom).} \label{justvel_vort}
\end{figure}

\section{Conclusions}

We have analyzed a VV scheme for NSE enhanced with CDA, using linearized backward Euler or BDF2 in time and finite elements in space.  We proved that applying CDA preserves the unconditional stability properties of the scheme, and also yields optimal long-time accuracy if both velocity and vorticity are nudged, or velocity-only.  If only velocity is nudged, then the convergence in time to the true solution is slower, but still exponentially fast in time.   Numerical tests illustrate the theory, including the difference between nudging only velocity or also nudging vorticity.

For future directions, since nudging vorticity is difficult in practice due to accurate measurement data not typically being available, one may try to obtain better results for the velocity-only-nudging by penalizing the difference between $w_h$ and $\rot v_h$ in the vorticity equation.  That is, by setting $\mu_2=0$ and adding the term $\gamma (w - \rot u)$ to the vorticity equation  \eqref{VVda}, it may be possible to analytically prove a convergence result resembling Theorem \ref{nlthmL2}.  Determining whether this is possible, and if so then for what values of $\gamma$, and whether it works in practice (i.e. how large are associated constants), would need a detailed further study which the authors plan to undertake.

\bibliographystyle{abbrv}
\bibliography{graddiv}

\end{document}